\documentclass[11pt]{amsart}
\usepackage[a4paper,includeheadfoot,margin=1in]{geometry}
\usepackage[english]{babel}
\usepackage{amsmath,amssymb,bm}
\usepackage{amsthm,thmtools}
\usepackage{mathtools}
\usepackage{enumerate}
\usepackage{enumitem}
\usepackage{xcolor}
\usepackage{lmodern,microtype}
\usepackage{relsize} % temporary (for comments); remove in final version!
\usepackage{cite}
\usepackage{graphicx}
\usepackage{calc}
\usepackage{comment}
\usepackage[colorlinks=true,linkcolor=,citecolor=,urlcolor=blue]{hyperref} % should be the last package to be loaded, with a small number of exceptions; see https://tex.stackexchange.com/q/1863
\usepackage[capitalize]{cleveref} % should be loaded AFTER hyperref

\title{Free Inhomogeneous Wreath Product of Quantum Groups}
\author[J. \MakeLowercase{\textsc{van}} Dobben de Bruyn]{Josse \MakeLowercase{\textsc{van}} Dobben de Bruyn}
\address{Technical University of Denmark}
\curraddr{Charles University}
\email{josse.van-dobben-de-bruyn@matfyz.cuni.cz}
\thanks{Josse van Dobben de Bruyn, Prem Nigam Kar, David E. Roberson and Peter Zeman are supported by Carlsberg Semper Ardens Accelerate CF21-0682 Quantum Graph Theory.}
\author[A. Freslon]{Amaury Freslon}
\address{Université Paris-Saclay}
\email{amaury.freslon@universite-paris-saclay.fr}
\author[P.N. Kar]{Prem Nigam Kar}
\address{Technical University of Denmark}
\email{pkar@dtu.dk}
\author[D.E. Roberson]{David E. Roberson}
\address{Technical University of Denmark}
\address{QMATH, University of Copenhagen}
\email{dero@dtu.dk}
\author[P. Zeman]{Peter Zeman}
\address{Technical University of Denmark}
\address{Department of Algebra, Faculty of Mathematics and Physics, Charles University}
\email{peter.zeman@matfyz.cuni.cz}
\thanks{Peter Zeman was funded by the European Union (ERC, POCOCOP, 101071674). Views and opinions expressed are however those of the author(s) only and do not necessarily reflect those of the European Union or the European Research Council Executive Agency. Neither the European Union nor the granting authority can be held responsible for them.}


\newcommand{\fwr}{\mathbin{\wr_\ast}}
\newcommand{\gwr}{\mathbin{\wr\wr}}
\newcommand{\fgwr}{\mathbin{\gwr_\ast}}

\DeclarePairedDelimiter\card{\lvert}{\rvert}
\DeclareMathOperator{\Qut}{Qut}
\DeclareMathOperator{\Aut}{Aut}
\DeclareMathOperator{\QIso}{QIso}
\DeclareMathOperator{\rel}{rel}
\DeclareMathOperator{\lev}{lev}
\DeclareMathOperator{\dist}{dist}
\DeclareMathOperator{\wl}{WL}
\DeclareMathOperator{\conn}{conn}

\newcommand{\bbC}{{\mathbb{C}}}
\newcommand{\bbF}{{\mathbb{F}}}
\newcommand{\bbG}{{\mathbb{G}}}
\newcommand{\bbH}{{\mathbb{H}}}
\newcommand{\bbN}{{\mathbb{N}}}
\newcommand{\bbS}{{\mathbb{S}}}
\newcommand{\bbX}{{\mathbb{X}}}

\def\calA{{\mathcal A}} \def\calB{{\mathcal B}} \def\calC{{\mathcal C}}
  
\def\calG{{\mathcal G}}

\def\calS{{\mathcal S}} \def\calT{{\mathcal T}}

\DeclareFontFamily{U}{mathb}{\hyphenchar\font45}
\DeclareFontShape{U}{mathb}{m}{n}{
<5> <6> <7> <8> <9> <10> gen * mathb
<10.95> mathb10 <12> <14.4> <17.28> <20.74> <24.88> mathb12
}{}
\DeclareSymbolFont{mathb}{U}{mathb}{m}{n}
\DeclareMathSymbol{\bigastglyph}{2}{mathb}{"06}
\DeclareMathOperator*{\normalbigast}{\bigastglyph}

\makeatletter
\DeclareRobustCommand\bigop[2][\displaystyle]{%
  \mathop{\vphantom{#1\sum}\mathpalette\bigop@{#2}}\slimits@
}
\newcommand{\bigop@}[2]{%
  \vcenter{%
    \sbox\z@{$#1\sum$}%
    \hbox{\resizebox{\ifx#1\displaystyle.9\fi\dimexpr\ht\z@+\dp\z@}{!}{$\m@th#2$}}%
  }%
}
\makeatother
\newcommand{\displaybigast}{\DOTSB\bigop{\bigastglyph}}
\newcommand{\textbigast}{\DOTSB\bigop[\textstyle]{\bigastglyph}}

\DeclareMathOperator*{\bigast}{\mathchoice{\displaybigast}{\textbigast}{\normalbigast}{\normalbigast}}
\makeatletter
\def\oversortoftilde#1{\mathop{\vbox{\m@th\ialign{##\crcr\noalign{\kern3\p@}%
      \sortoftildefill\crcr\noalign{\kern3\p@\nointerlineskip}%
      $\hfil\displaystyle{#1}\hfil$\crcr}}}\limits}
\def\sortoftildefill{$\m@th \setbox\z@\hbox{$\braceld$}%
  \braceld\leaders\vrule \@height\ht\z@ \@depth\z@\hfill\braceru$}
\makeatother

\newlength{\depthofsumsign}
\newlength{\totalheightofsumsign}
\newlength{\heightanddepthofargument}

\makeatletter
\newcommand*{\DivideLengths}[2]{%
  \strip@pt\dimexpr\number\numexpr\number\dimexpr#1\relax*65536/\number\dimexpr#2\relax\relax sp\relax
}
\makeatother

\DeclareSymbolFont{bbold}{U}{bbold}{m}{n}
\DeclareSymbolFontAlphabet{\mathbbold}{bbold}
\newcommand{\one}{\ensuremath{\mathbbold{1}}}

\declaretheorem[style=definition,numberwithin=section]{definition}
\declaretheorem[style=definition,numberlike=definition]{construction}

\declaretheorem[style=plain,numberlike=definition]{theorem}
\declaretheorem[style=plain,numberlike=definition]{lemma}
\declaretheorem[style=plain,numberlike=definition]{proposition}
\declaretheorem[style=plain,numberlike=definition]{corollary}

\declaretheorem[style=definition,numberlike=definition]{remark}

\colorlet{Amaury_colour}{blue!80!black}
\colorlet{David_color}{purple!80!black}
\colorlet{Josse_colour}{orange!70!black}
\colorlet{Prem_colour}{red!80!black}
\colorlet{Peter_colour}{green!75!black}
\colorlet{TODO_colour}{pink!80!black}

\begin{document}

\begin{abstract}
	We introduce the free inhomogeneous wreath product of compact matrix quantum groups, which generalizes the free wreath product (Bichon 2004).
	We use this to present a general technique to determine quantum automorphism groups of connected graphs in terms of their maximal biconnected subgraphs, provided that we have sufficient information about their quantum automorphism groups.
	We show that this requirement is met for outerplanar graphs, leading to algorithms to compute the quantum automorphism groups of these graphs, as well as recovering results for forests and block graphs.
\end{abstract}

\maketitle

\section{Introduction}
In \cite{bichon2004free}, Bichon defined the \emph{free wreath product} $\bbG \wr_* \bbS_n^+$ of a compact quantum group $\bbG$ by the quantum symmetric group $\bbS_n^+$.
This is the quantum analogue of the wreath product of classical groups.
The free wreath product of a compact quantum group $\bbG$ by a quantum permutation group $\bbH \leq \bbS_n^+$ is characterized by the following property: whenever $\bbG$ acts on a quantum set $\bbX$, the free wreath product $\bbG \wr_* \bbH$ acts on $n$ disjoint copies of $\bbX$ by independently acting with $\bbG$ on each copy of $\bbX$, while also acting with $\bbH$ to quantum permute the different copies of $\bbX$.

Apart from receiving interest from researchers in quantum groups and operator algebras (e.g.~\cite{Banica-Vergnioux,Tarrago-Wahl,Fima-Troupel-OA,Brownlowe-Robertson}), the free wreath product also has several applications, in particular to quantum automorphism groups of graphs \cite{bichon2004free,Q_Aut_Trees,Quantum_Sabidussi,meunier2023quantum,Brownlowe-Robertson} and free probability \cite{Banica-et-al-Bessel,Lemeux-Tarrago}.
However, in other settings, the free wreath product is too restrictive and cannot be used.
Several extensions and generalizations of the free wreath product have been given; see for instance \cite{Fima-Pittau,Freslon-Skalski,Freslon-amalgamation,Fima-Troupel-generalized}.

In this paper, we present another generalization that we call the \emph{free inhomogeneous wreath product}, named after its classical counterpart \cite{inhom_wr_prod}.
Here we have a slightly more general setting than the one outlined above.
We still have a quantum permutation group $\bbH \leq \bbS_n^+$, but now we have for each orbit $\Omega_i \subseteq [n]$ of $\bbH$ a separate compact matrix quantum group $\bbG_i$.
The free inhomogeneous wreath product is then a compact quantum group, with the following property: if all the quantum groups $\bbG_i$ act on finite (classical) sets $X_i$, the free inhomogeneous wreath product $(\bbG_1, \ldots, \bbG_m) \fgwr \bbH$ acts on the disjoint union $\bigsqcup_{i=1}^m \bigsqcup_{\alpha \in \Omega_i} X_i$ by independently acting with the respective $\bbG_i$ on the individual copies of the $X_i$, while also acting with $\bbH$ to quantum permute the different copies of the $X_i$.
In other words, we generalize the free wreath product by allowing the quantum group $\bbG$ to differ between different orbits of $\bbH$.

Our first result is the definition of the free inhomogeneous wreath product ${(\bbG_1, \ldots, \bbG_m) \fgwr \bbH}$, and the proof that it is a compact quantum group.
We note that it contains two well-known products of quantum groups as special cases: the free wreath product (by setting $\bbG_1 = \cdots = \bbG_m$), but also the free product (by setting $\bbH = 1$).
Therefore, it is not just another generalization of the free wreath product, but also a new product of quantum groups that is interesting in its own right.

After defining the free inhomogeneous wreath product and proving some basic properties, the remainder of this paper is dedicated to applications in graph theory.

First, in \cref{subsec:disjoint-union}, we revisit disjoint unions of graphs.
If $X_1,\ldots,X_n$ is a finite collection of graphs such that for all $i,j \in [n]$, the graphs $X_i$ and $X_j$ are either isomorphic or not quantum isomorphic%
	%%%% FOOTNOTE
	\footnote{When some of the pairs $X_i,X_j$ are quantum isomorphic but not isomorphic, $\Qut(X_1 \sqcup \cdots \sqcup X_n)$ depends not only on $\Qut(X_1),\ldots,\Qut(X_n)$ but also on the quantum isomorphism algebras $\QIso(X_i,X_j)$, so the free inhomogeneous wreath product cannot be used.
	See \autoref{rmk:quantum-isomorphic-pair} and \cite[Theorem 5.3]{Quantum_Sabidussi} for more details.}%
	%%%% END FOOTNOTE
, then we show that $\Qut(X_1 \sqcup \cdots \sqcup X_n)$ is isomorphic to a free inhomogeneous wreath product of the $\Qut(X_i)$ by the quantum permutation group that quantum permutes the sets of isomorphic graphs (among the $X_i)$.
A precise statement of this result can be found in \autoref{thm:disjoint-union}.

Next, we move on to a more advanced application.
Apart from the decomposition of a graph into its connected components, every connected graph can be further decomposed into its \emph{blocks} (also known as \emph{biconnected components}).
Every edge in a graph belongs to exactly one block, and the blocks and cut vertices together form a tree, called the \emph{block tree} (or \emph{block-cut tree}).
Using a mix of combinatorial and quantum group techniques, we show that the quantum automorphisms of a connected graph $X$ quantum permute the blocks of $X$ and restrict to quantum automorphisms of the block tree of $X$. This provides an alternate approach to some of the results of \cite{freslon2025block}.

As a consequence, we prove that $\Qut(X)$ can be described recursively in terms of free inhomogeneous wreath products of the quantum automorphism groups of the blocks further down the block tree, provided that the rooted subgraphs of $X$ that appear in the process are either isomorphic or not quantum isomorphic. 

Finally, we exhibit some examples of graphs for which this condition is met.
In doing so, we derive algorithms to compute the quantum automorphism groups of two classes of graphs: \emph{block graphs} (in which every block is a clique) and \emph{outerplanar graphs} (that can be drawn in the plane without crossing edges in such a way that all vertices lie on the outer face).
This is one of the largest class of graphs for which we know how to compute the quantum automorphism groups; previously such techniques were known for forests \cite{Q_Aut_Trees,meunier2023quantum}, tree-cographs \cite{meunier2023quantum}, and certain particular families of graphs (e.g. \cite{Schmidt-distance-transitive, gromada2023quantum}). At the same time as the present work, quantum automorphism groups of block-cographs (a class which strictly contains block graphs) were also computed in \cite{FMPbis2025} as part of a general study of quantum properties of 0-hyperbolic graphs.

\section{Preliminaries}

\subsection{Graph theory}
\label{subsec:prelim:graphs}

Throughout this article, all graphs are assumed to be finite and \emph{simple} (no parallel edges and no self-loops).
For a graph $X$, we denote the number of connected components by $\conn(X)$.
A \emph{cut vertex} is a vertex $v \in V(X)$ such that $\conn(X \setminus \{v\}) > \conn(X)$.
A \emph{block} (or \emph{biconnected component}) is a maximal subgraph $B \subseteq X$ that is connected and remains connected after removing any one vertex of $B$.
Every edge of $X$ belongs to exactly one block, and two edges belong to the same block if and only if they lie on a common cycle (see e.g.~\cite[\S{}5.2]{Bondy-Murty}).
A graph $X$ is called \emph{biconnected} if it has exactly one block.

Assume now that $X$ is connected.
Let $\calB(X)$ be the set of blocks of $X$ and let $\calC(X)$ be the set of cut vertices in $X$.
The \emph{block tree $T(X)$ of $X$} is a bipartite graph with vertex set
$$V(T(X)) = \calB(X) \sqcup \calC(X)$$
and there is an edge $\{B,v\} \in E(T(X))$ if $v \in V(B)$.
It is easy to see that $T(X)$ is in fact a tree.
We refer to the vertices $V(T(X))$ of $T(X)$ as \emph{nodes} to avoid confusion.
There are two types of nodes: \emph{block nodes} $B \in \calB(X)$ and \emph{cut nodes} $v \in \calC(X)$.

The \emph{distance} $\dist(v,w)$ between two vertices is the length of a shortest path between $v$ and $w$, where $\dist(v,w) = \infty$ if $v$ and $w$ belong to different connected components.
The \emph{eccentricity} of a vertex $v \in V(X)$ is defined as $\epsilon(v) := \max_{w \in V(X)} \dist(v,w)$; that is, the maximum distance from $v$ to any other vertex in $V(X)$.
The \emph{(Jordan) center} $Z(X)$ of a graph $X$ is the set of vertices with minimum eccentricity.

It is well known that the Jordan center of a tree $T$ always has size $1$ or $2$, and that it coincides with the midpoint(s) of every longest path in $T$.
In the case of a block tree $T(X)$ of a connected graph $X$, the center $Z(T(X))$ always has size $1$.
To see this, note that every cut vertex in $X$ belongs to at least two blocks, so all leaves of $T(X)$ are block nodes.
Therefore every longest path in $T(X)$ has an odd number of vertices (alternating between block nodes and cut nodes, starting and ending with a block node), so its midpoint is a single vertex.

If the center of the block tree $T$ is a cut node $v\in \calC(X)$, then we say that $v$ is the \emph{central cut vertex} of $X$.
If the center of the tree $T$ is a block node $B \in \calB(X)$, then we say that $B$ is the \emph{central block} of $X$.
We consider the tree $T$ to be a rooted tree, where the root is its center.
Given any node $N$ of the block tree $T$, the subtree of $T$ rooted at $N$ and consisting in all nodes of $T$ that, when starting at the root, require you to travel through $N$ to visit, defines a unique induced subgraph of $X$, which we denote by $X^{\leq N}$. 

For a block $B \in \calB(X)$, we say that a cut vertex $v\in V(B)\cap\calC(X)$ is a \emph{child cut vertex} of $B$ if $v$ is farther than $B$ from the root of the block tree $T$.
Otherwise, $v$ is a \emph{parent cut vertex} of $B$.
Note that a central block has only children cut vertices and a non-central block has a unique parent cut vertex.
Analogously, we define the \emph{child blocks} and the \emph{parent block} of a cut vertex.

\subsection{Compact Matrix Quantum Groups}  

In this subsection, we provide some of the necessary background on compact matrix quantum groups, referring the reader to~\cite{neshveyev_tuset, freslon2023compact} for a comprehensive exposition. We begin with the definition of a compact quantum group. A \emph{compact quantum group} $\bbG$ is an ordered pair $(\mathcal{A},\Delta)$, where $\mathcal{A}$ is a unital $C^*$\nobreakdash-algebra and $\Delta: \mathcal{A} \to \mathcal{A} \otimes \mathcal{A}$ is a unital $*$\nobreakdash-homomorphism (the symbol $\otimes$ will always denote the minimal tensor product of C*-algebras), known as the \emph{comultiplication} of $\bbG$, satisfying the following conditions:
	\begin{enumerate}
		\item \emph{co-associativity:} $(\Delta \otimes id) \Delta = (id \otimes \Delta) \Delta$;
		\item \emph{cancellation property:} the sets $\Delta(\mathcal{A})(1 \otimes \mathcal{A}) $ and $\Delta(\mathcal{A})(\mathcal{A} \otimes 1)$ are dense in $\mathcal{A} \otimes \mathcal{A}$.
	\end{enumerate}

Let $G$ be a compact group and $C(G)$ denote the set of all continuous complex-valued functions on $G$.
Then, one can define a $\ast$-homomorphism $\Delta$ from $C(G)$ to $C(G) \otimes C(G) \cong C(G \times G)$ as follows: $\Delta(f)(g,h) = f(gh)$, for $(g,h) \in G\times G$.
Then the ordered pair $(C(G), \Delta)$ is a compact quantum group, and both the set $G$ and the multiplication of $G$ can be recovered from $(C(G), \Delta)$.
Moreover, every compact quantum group $(\calA, \Delta)$ with $\calA$ commutative is of this type.
In light of this example, given a compact quantum group $\bbG = (\calA, \Delta)$, it is customary to denote the $C^*$-algebra $\calA$ as $C(\bbG)$. 

We will be interested in a specific family of compact quantum groups, which we now introduce.

\begin{definition}
    Let $C(\bbG)$ be a unital $C^*$\nobreakdash-algebra generated by $\{u_{ij}:i,j\in [n]\}$ such that the matrices $[u_{ij}]_{i,j=1}^n$ and $[u_{ij}^*]_{i,j=1}^n$ are invertible and $\Delta: C(\bbG) \to C(\bbG) \otimes C(\bbG)$ is a unital $*$\nobreakdash-homomorphism satisfying
	\begin{equation}
		\Delta(u_{ij}) = \sum_{k=1}^n u_{ik}\otimes u_{kj}, \qquad i,j\in [n]. \label{eqdef:comult-map}
	\end{equation}
	Then, $\bbG = (C(\bbG), \Delta)$ forms a compact quantum group (see {\cite[Proposition 1.1.4]{neshveyev_tuset}}). Such compact quantum groups are known as \emph{compact matrix quantum groups}. They are usually specified by the pair $\bbG = (C(\bbG), u)$, as the comultiplication only depends on $u$. The matrix $u$ will be called the \emph{fundamental representation} of the compact matrix quantum group $\bbG$.
\end{definition}

The prototypical example of a compact matrix quantum group is a compact matrix group. Indeed, given a compact matrix group $G\subseteq GL_n(\bbC)$, the set of continuous complex-valued functions on $G$ is generated by the functions $\{u_{ij}\}_{i,j \in [n]}$, where $u_{ij}(g)$ is the $(i,j)$-th coordinate of the matrix $g\in G$. It is easy to see that the matrices $[u_{ij}]_{i,j=1}^n$ and $[u_{ij}^*]_{i,j=1}^n$ are invertible. One can also show that $\Delta: C(G) \to C(G) \otimes C(G)$ defined by $u_{ij} \to \sum_{k} u_{ik} \otimes u_{kj}$ is a $\ast$-homomorphism, thus turning $(C(G),u)$ into a compact matrix quantum group. One can also show that any compact matrix quantum group $\bbG = (C(\bbG), \Delta)$, where $C(\bbG)$ is commutative, is of this type. 

We now give an example of a genuinely quantum compact matrix quantum group; that is, one with a non-commutative $C^*$-algebra 
This example is particularly important to us because the quantum automorphism group of a graph, defined in the next subsection, is a quotient of this quantum group. It was introduced in~\cite{Wang}.

\begin{definition}
	\label{def:comp-mat-quantum-group}

	Let $C(\bbS_n^+)$ be the universal $C^*$-algebra generated by $\{u_{ij}\}_{i,j \in [n]}$ satisfying the following conditions:
	\begin{enumerate}[label = \roman*.]
		\item $u_{ij} = u_{ij}^* = u_{ij}^2$, i.e., the $u_{i,j}$ are self-adjoint projections. 
		\item $\sum_{j'=1}^n u_{ij'} = \sum_{i'=1}^n u_{i'j} = 1$ for all $i,j \in [n]$. 
	\end{enumerate}
	Then, $\bbS_n^+ = (C(\bbS_n^+), [u_{ij}])$ is a compact matrix quantum group known as the \emph{quantum symmetric group} acting on $n$ points.
\end{definition}

It is also shown in~\cite{Wang} that $\bbS_n^+$ is the universal quantum group acting on $n$ points.
Given any $C^*$-algebra $\calA$, a matrix $u = [u_{ij}] \in M_n(\mathcal{A})$ satisfying \textit{i, ii} in \cref{def:comp-mat-quantum-group} is known as a \emph{magic unitary}.
A \emph{quantum permutation group} acting on $n$ points, is a compact matrix quantum group $\bbG = (C(\bbG), v)$, such that the fundamental representation $v \in M_n(C(\bbG))$ is a magic unitary. Similarly, when we write a quantum permutation group acting on a set $\Omega$, we mean a quantum permutation group acting on $\left\lvert \Omega\right\rvert$.

For such a quantum permutation group $\bbG$, by the universal property of $C(\bbS_n^+)$, we see that there is a surjective $*$-homomorphism $\pi: C(\bbS_n^+) \to C(\bbG)$ mapping $u_{ij}$ to $v_{ij}$ for all $i,j \in [n]$.
In addition to this, if we denote the comultiplications of $\bbG$ and $\bbS_n^+$ by $\Delta_\bbG$ and $\Delta_{\bbS_n^+}$ respectively, then
\begin{equation*}
    \Delta_\bbG \circ \rho = ( \rho \otimes \rho ) \circ \Delta_{\bbS_n^+}.
\end{equation*}
In other words, $\bbG$ is a \emph{quantum subgroup} of $\bbS_n^+$. 

The classical notion of orbits and orbitals can be extended to quantum permutation groups in a natural way. Given a quantum permutation group $\bbG = (C(\bbG), u)$ acting on $n$ points,~\cite{lupini2020nonlocal} introduced the following relations: 
\begin{itemize}
    \item $i \sim_1 j$ if $u_{ij} \neq 0$, for $i,j \in [n]$ and
    \item $(i,j) \sim_2 (k,l)$ if $u_{ik}u_{jl} \neq 0$ for $(i,j), (k,l) \in [n] \times [n]$.
\end{itemize}
These can be shown to be equivalence relations; the equivalence classes of $[n]$ induced by $\sim_1$ are known as the \emph{orbits} of $\bbG$ and the equivalence classes of $[n] \times [n]$ induced by $\sim_2$ are known as the \emph{orbitals} of $\bbG$.

\subsection{Quantum Automorphism Groups of Graphs}
This subsection is dedicated to introducing quantum automorphism groups of graphs, which will serve as the primary avenue for application of the free inhomogeneous wreath product. 

Let $X$ be a finite graph on $n$ vertices. The quantum automorphism group of $X$, denoted by $\Qut(X)$, was defined by Banica~\cite{banica2005quantum} to be the quantum permutation group $\Qut(X) = (C(\Qut(X)), u)$, where $C(\Qut(X))$ is the universal $C^*$-algebra generated by $\{u_{ij}\}_{i,j \in [n]}$ subject to the following conditions: 
\begin{enumerate}[label = \roman*.]
    \item $u$ is a magic unitary and
    \item $A_X u = u A_X$, where $A_X$ is the adjacency matrix of $X$.
\end{enumerate}

If $X$ is a vertex coloured graph with colouring $c: V(X) \to \{1,2,\dots,k\}$, for some $k \in \bbN$, then the quantum automorphism group of the coloured graph $X$, denoted by $\Qut_c(X)$, is the quantum permutation group $\Qut_c(X) = (C(\Qut_c(X)), u)$, where $C(\Qut_c(X))$ is the universal $C^*$-algebra generated by $\{u_{ij}\}_{i,j \in }$ such that: 
\begin{enumerate}[label = \roman*.]
    \item $u$ is a magic unitary,
    \item $A_X u = u A_X$, and
    \item $u_{ij} = 0$, whenever $c(i) \neq c(j)$.
\end{enumerate}

We will have to check several times in the sequel the commutation relation with the adjacency matrix, and we now give an equivalent property which will prove more practical. Using notations from \cite{lupini2020nonlocal}, we write $\mathrm{rel}(i, k) = \mathrm{rel}(j, l)$ if one of the following holds : $i = k$ and $j = l$, $i\sim k$ and $j\sim l$, $i\nsim k$ and $j\nsim l$. 

Let $X$ be a coloured graph, with colouring $c: V(X)\to [m]$, and let $x \in V(X)$, be a vertex. By the \emph{quantum stabiliser} of $x$, we mean the quantum subgroup $\Qut_c(X)_x$ of $\Qut(X)$, which is the quantum automorphism group of the coloured graph $X$, $\Qut_{c'}(X)$, with colouring $c': V(X) \to [m+1]$ defined as follows: 
\begin{align*}
    c(y) = \begin{cases}
        m+1 & \text{ if } y = x, \\
        c(y) & \text{ otherwise.}
    \end{cases}
\end{align*}
In other words, if $u$ is the fundamental representation of $\Qut_c(X)$, $C(\Qut_c(X)_x)$ is the quotient of $C(\Qut_c(X))$ by the relation $u_{xx}=1$, and the fundamental representation of $\Qut_c(X)_x$ is $[q(u_{xy})]_{x,y\in V(X)}$, where $q: C(\Qut_c(X)) \to C(\Qut_c(X)_x)$ is the quotient map. In particular, it should be noted that if $u_{xx}=1$, then $\Qut_c(X)_x \cong \Qut_c(X)$. 
\begin{lemma}
	\label{lem:commutation_adjacency}
	Let $u$ be a magic unitary and let $X$ be a graph.
	Then, $uA_{X} = A_{X}u$ if and only if, for any four vertices $i, j, k, l$, we have
	\begin{equation*}
		u_{ij}u_{kl} = 0
	\end{equation*}
	whenever $\mathrm{rel}(i, k)\neq \mathrm{rel}(j, l)$.
\end{lemma}

Given a (coloured) graph $X$, we shall sometimes refer to the orbits/orbitals of $\Qut(X)$ as the \emph{quantum orbits/orbitals of $X$}.

We now give a brief overview of the Weisfeiler-Leman algorithm, which will be used frequently in this article as a heuristic for quantum isomorphism of graphs. 

Given a graph $X$, the Weisfeiler--Leman algorithm \cite{Weisfeiler} constructs a colouring $\wl_X : V(X) \times V(X) \to C$ via the following iterative procedure:

First, construct the initial colouring
\[ c_0(x, y) \coloneqq
	\begin{cases}
		0 \quad \text{if $x = y$,}\\
		1 \quad \text{if $xy \in E(X)$,}\\
		2 \quad \text{if $xy \in E(\overline{X})$.}
	\end{cases} \]
Then, we repeat the iterate procedure.
Given a colouring $c_k : V(X) \times V(X) \to C$, for every pair of colours $i,j \in C$ we define
\[ \Delta_{ij}(x,y) \coloneqq |\{z \in V(X) : c(x, z) = i \text{ and } c(z, y) = j\}|. \]
Then the \emph{refinement} of $c_k$ is the colouring $c_{k+1}$ given by
\[ c_{k+1}(x,y) \coloneqq (c_k(x,y), (\Delta_{ij}(x,y))_{i,j\in C}). \]

Note that each step refines the initial colouring $c_0$. Hence, after a finite number of steps we end up with a colouring $c_l$ that is \emph{stable} in the sense that $c_l$ and $c_{l+1}$ induce the same partition of $V(X) \times V(X)$.
Once we have reached this stable colouring, the algorithm returns this colouring and terminates.
The stable colouring $\overline{c}$ returned by the Weisfeiler--Leman algorithm will be called the \emph{stable colouring of $X$} and be denoted by $\wl_X$.

The following result from \cite{lupini2020nonlocal}, says that if two ordered pairs of vertices in $V(X)\times V(X)$ are distinguished by the Weisfeiler-Leman algorithm \cite{Weisfeiler}, then they lie in different orbitals of $\Qut(X)$.

\begin{lemma}[{\cite[Corollary 3.13]{lupini2020nonlocal}}]
	\label{lem:wl-quantum-automorphims}
	Let $X$ be a graph and let $u$ denote the fundamental representation of $\Qut(X)$.
	Then, $\wl_X(p,p')\neq \wl_X(q,q')$ implies $u_{pq}u_{p'q'} = 0$.
\end{lemma}

We now introduce quantum isomorphism of graphs, which is a topic that is closely related to quantum automorphism groups of graphs. The original definition of quantum isomorphism of graphs was in terms of existence of a perfect quantum strategy for the graph isomorphism game introduced in \cite{atserias2019quantum}. We shall be working with an equivalent definition introduced in \cite{lupini2020nonlocal}. We refer the reader to \cite{atserias2019quantum, lupini2020nonlocal} for a detailed exposition of these topics.

\begin{definition}
Two graphs $X$ and $Y$ are said to be \emph{quantum isomorphic}, written as $X \cong_q Y$, if there is a non-zero unital $C^*$\nobreakdash-algebra $\mathcal{A}$ and a magic unitary $u = [u_{xy}]_{x \in V(X), y \in V(Y)}$ with entries from $\mathcal{A}$ such that $A_Xu = uA_Y$.
\end{definition}

Given graphs $X$ and $Y$, along with vertices $x \in V(X)$ and $y \in V(Y)$, we write $X_x \cong Y_y$, if there is an isomorphism $\sigma: V(X) \to V(Y)$ such that $\sigma(x) = y$. Similarly, we write $X_x \cong_q Y_y$, if there is a non-zero unital $C^*$\nobreakdash-algebra $\mathcal{A}$ and a magic unitary $u = [u_{x'y'}]_{x' \in V(X), y' \in V(Y)}$ with entries from $\mathcal{A}$ such that $A_Xu = uA_Y$ and $u_{xy} = 1$. 

The following result from \cite{lupini2020nonlocal} establishes a strong connection between the theory of quantum automorphism groups of graphs and the notion of quantum isomorphism: 

\begin{theorem}[{\cite[Theorem 4.5]{lupini2020nonlocal}}]
    Two connected graphs $X$ and $Y$ are quantum isomorphic if and only if there are vertices $x \in V(X)$ and $y \in V(Y)$ that are in the same orbit of $\Qut(X \sqcup Y)$.
\end{theorem}

The above result is a quantum analogue of the fact that two connected graphs are isomorphic if and only if there are vertices $x \in V(X)$ and $y \in V(Y)$ that are in the same orbit of $\Aut(X \sqcup Y)$.

\subsection{The Free Product and The Free Wreath Product} In this subsection, we introduce two products of compact matrix quantum groups: the free product and the free wreath product. We will also discuss how they can be used to describe quantum automorphism groups of disjoint unions. We first introduce the free product of compact matrix quantum groups, which is an analogue of the direct product of groups.
\begin{definition}
	Let $\bbG = (C(\bbG), u)$ and $\bbH = (C(\bbH), v)$ be compact matrix quantum groups. Then, their \emph{free product} $\bbG \ast \bbH$ is defined as the compact matrix quantum group $(C(\bbG) \ast_{\bbC} C(\bbH), u \oplus v)$.
\end{definition}

It is not too difficult to see that the free product of two quantum permutation groups is again a quantum permutation group. We can state the following general result relating quantum automorphism groups of disjoint unions of non quantum isomorphic graphs and free products of compact matrix quantum groups:

\begin{theorem}[{\cite[Lemma 6.4]{Q_Aut_Trees}}]
	Let $X_1, \dots, X_n$ be vertex coloured graphs such that for any $i \ne j$, no connected component of $X_i$ is quantum isomorphic to a connected component of $X_j$. Then,
	\begin{equation*}
		\Qut_c\left(\bigsqcup_{i=1}^n X_i\right) = \bigast_{i=1}^n\Qut_c(X_i) 
	\end{equation*}
	where $\bigsqcup_{i=1}^n X_i$ denotes the disjoint union of $X_1, \dots, X_n$.
\end{theorem}

Next, we introduce the free wreath product of compact matrix quantum groups, first introduced in \cite{bichon2004free}. This is a quantum analogue of the wreath product of groups. 

\begin{definition}
	\label{def:free-wreath-product}
	Let $\bbG=(C(\bbG), [u_{ij}]_{i,j \in [m]})$ and $\bbH=(C(\bbH), [v_{ab}]_{a,b \in [n]})$ be two compact matrix quantum groups. The \emph{free wreath product of $\bbG$ and $\bbH$}, denoted by $\bbG \wr_* \bbH$, is the compact matrix quantum group
	\[ C(\bbG \wr_* \bbH) \coloneqq (C(\bbG^{*n}) \ast_\bbC C(\bbH))/ \langle [u_{ij}^{(a)}, v_{ab}] = 0 \mid i,j \in [m], \ a,b \in [n] \rangle , \]
	where $u^{(a)}$ denotes the $a$-th diagonal block of the magic unitary $u^{(1)} \oplus \cdots \oplus u^{(n)}$ of $\bbG^{*n}$, with fundamental representation $[w_{(a,i)(b,j)}]_{(a,i),(b,j) \in [n] \times [m]}$ (of $\bbG \wr_* \bbH$) given by
	\[ w_{(a,i)(b,j)} \coloneqq u_{ij}^{(a)}v_{ab}. \]
\end{definition}

We have the following result on quantum automorphism groups of disjoint unions of multiple copies of a connected graph:

\begin{theorem}[{\cite[Theorem 6.1]{BanicaBichon}}]
	Let $X$ be a connected vertex coloured graph and $n \in \bbN$. Let $\bigsqcup_{i=1}^n X$ denote the disjoint union of $n$ copies of $X$, all with the same colouring. Then, $\Qut_c(\bigsqcup_{i=1}^n X) = \Qut_c(X) \wr_* \bbS_n^+$, where $\wr_*$ denotes the free wreath product.
\end{theorem}

Note that \cite{BanicaBichon} works with edge-coloured graphs, but the proof is easily adapted to the vertex-coloured case. The free wreath product is a useful tool when describing quantum automorphism groups of lexicographic products of graphs. We refer the reader to~\cite{Quantum_Sabidussi} for further details.

\section{The free Inhomogeneous Wreath Product}

In this section, we will introduce our main tool for describing the quantum automorphism groups of connected graphs in terms of the quantum automorphism groups of their blocks.
It is a generalization of the classical \emph{inhomogeneous wreath product}, which we first briefly recall.

If $H$ is a group and $\Omega$ is an $H$-set (that is, a set acted upon by $H$), we denote the \emph{orbit} of $x\in\Omega$ under $H$ by the set $x^{H} \coloneq \{x^{h} : h \in H\}$.
These form a partition $\Omega_{1}\cup \dots\cup \Omega_{m}$ of $\Omega$.
The inhomogeneous wreath product, introduced in~\cite{inhom_wr_prod}, is a generalization of the wreath product, where different groups are allowed on the left side, given by the orbits of an action on the right side.

\begin{definition}[\cite{inhom_wr_prod}]
	Let $G_{1}, \dots, G_{m}$ and $H$ be groups, let $\Omega$ be finite $H$-set with orbits $\Omega_{1}, \dots, \Omega_{m}$, and set
	\begin{equation*}
		K = \prod_{i=1}^{m}\prod_{\alpha\in\Omega_i}G_{i,\alpha}
	\end{equation*}
	where $G_{i,\alpha} = G_{i}$ for all $\alpha \in \Omega_{i}$.
	The \emph{inhomogeneous wreath product}
	\begin{equation*}
		(G_{1}, \dots, G_{m})\gwr H
	\end{equation*}
	is the semidirect product of $K$ by $H$, where $H$ acts on $K$ by $(d_{i,\alpha})_{i\in [m], \alpha\in \Omega_{i}}^{q} \coloneq (d_{i,\alpha^{q}})_{i\in [m], \alpha\in \Omega_{i}}$.
\end{definition}

Our goal is now to define a compact quantum group analogue of this, following the construction of the free wreath product by Bichon in \cite{bichon2004free}.
We also refer the reader to \cite[Sec 7.2.2]{freslon2023compact} for a detailed treatment.

\begin{construction}
    \label{cstr:FIWP}
    Let $\bbG_{1}, \dots, \bbG_{m}$ be compact matrix quantum groups with fundamental representations $(g^{(i)}_{pq})_{pq\in \Lambda_{i}}$, and let $\bbH$ be a quantum permutation group acting on a finite set $\Omega$ with orbits $\Omega_{1}, \dots, \Omega_{m}$.
    For each $1\leqslant i\leqslant m$ and $\alpha\in \Omega_{i}$, let $\bbG_{i,\alpha}$ be an isomorphic copy of $\bbG_{i}$ with fundamental representation $[g_{pq}^{(i,\alpha)}]_{p,q \in \Lambda_i}$.
    Let also $[h_{\alpha\beta}]_{\alpha, \beta\in \Omega}$ be a magic unitary fundamental representation of $\bbH$.
    Define the $C^{*}$-algebra $C(\bbF)$ to be the quotient of the free product
    \begin{equation*}
    	\mathcal{A} = \left(\bigast_{i=1}^{n} \left(\bigast_{\alpha\in\Omega_{i}} C(\bbG_{i,\alpha}) \right) \right) \ast C(\bbH)
    \end{equation*}
    by the relations
    \begin{equation*}
    	[g_{pq}^{(i,\alpha)},h_{\alpha\beta}] = 0 \quad \text{for all $i\in[m]$, $p,q\in\Lambda_{i}$, $\alpha\in\Omega_{i}$},
    \end{equation*}
    and let $\Lambda = \bigcup_{i=1}^{n}\Omega_{i}\times\Lambda_{i}$.
    Moreover, define the matrix $f = [f_{\alpha p,\beta q}]_{(\alpha, p), (\beta, q) \in \Lambda}$ as follows:
    \begin{equation*}
    	f_{\alpha p,\beta q} =
    	\begin{cases}
    		h_{\alpha\beta}g_{pq}^{(i,\alpha)} \quad &\text{if $\alpha,\beta \in \Omega_{i}$, for some $i\in[m]$,}\\
    		0 \quad &\text{otherwise.}
    	\end{cases}
    \end{equation*}
\end{construction}

As for usual free wreath products, the main subtelty of the construction is in the definition of the comultiplication.
We collect for convenience all the results that we need in the following statement.

\begin{theorem}
    \label{thm:FIWP}
	Let $\bbG_{1}, \dots, \bbG_{m}, \bbH$, $C(\bbF)$, and $f$ be as in \autoref{cstr:FIWP}.
	Then, the following hold true:
	\begin{enumerate}[label = \roman*.]
		\item There exists a $\ast$-homomorphism $\Delta: C(\bbF) \to C(\bbF)\otimes C(\bbF)$ satisfying the following:
		\[ \Delta(f_{\alpha p, \beta q}) = \sum_{\gamma r \in \Lambda}f_{\alpha p, \gamma r}\otimes f_{\gamma r, \beta q}. \]
		
		\item The pair $(C(\bbF), \Delta)$ is a compact quantum group,
		
		\item If all the matrices $g^{i}$ are magic unitaries, then so is the matrix $f$.
	\end{enumerate}
	In particular $(C(\bbF), \Delta)$ is a compact quantum group with fundamental representation $f$.
\end{theorem}
\begin{proof}
	We follow the lines of the original argument of Bichon in \cite{bichon2004free}.
	Let first $i_{\bbH}: C(\bbH) \to C(\bbF)$, be the injective embedding of $C(\bbH)$ in $C(\bbF)$ coming from the universal property of the free product and let $\Delta_{\bbH}: C(\bbH) \to C(\bbH) \otimes C(\bbH)$ be the comultiplication of $\bbH$.
	We define $\widetilde{\Delta}_\bbH: C(\bbH) \to C(\bbF) \otimes C(\bbF)$ as $\widetilde{\Delta}_{\bbH} = (i_{\bbH} \otimes i_{\bbH}) \circ \Delta_{\bbH}$.
	In particular, we have
	\begin{equation*}
		\widetilde{\Delta}_\bbH(h_{\alpha \beta}) = \sum_{\gamma \in \Omega} h_{\alpha \gamma} \otimes h_{\gamma \beta}
	\end{equation*}
	for all $\alpha, \beta \in \Omega$.
	
	We now want to define similarly a $*$-homomorphism $\widetilde{\Delta}_{\bbG_{i,\alpha}}: C(\bbG_{i,\alpha}) \to C(\bbF) \otimes C(\bbF)$, but this requires more care.
	First, we consider the $*$-homomorphism $\Phi_{\alpha, \beta}^{i} : C(\bbG_{i, \alpha}) \to C(\bbG_{i, \alpha})\otimes C(\bbG_{i, \beta})$, defined by
	\begin{equation*}
		\Phi_{\alpha, \beta}^{i}(g_{pq}^{i, \alpha}) = \sum_{r\in \Lambda_{i}}g_{pr}^{i, \alpha}\otimes g_{rq}^{i, \beta}.
	\end{equation*}
	Second, we observe that on $C(\bbG_{i, \alpha}) \otimes C(\bbG_{i, \beta}) \subset C(\bbF)\otimes C(\bbF)$, the map
	\begin{equation*}
		\Psi_{\alpha, \beta}^{i} : x\otimes y\mapsto (xh_{\alpha, \beta})\otimes y = (h_{\alpha, \beta}x)\otimes y
	\end{equation*}
	defines a $*$-homomorphism.
	Third, we set
	\begin{equation*}
		\Phi_{\alpha}^{i} = \sum_{\beta\in \Omega_{i}}\Psi_{\alpha, \beta}^{i}\circ\Phi_{\alpha, \beta}^{i} : C(\bbG_{i, \alpha})\to C(\bbF)\otimes C(\bbF).
	\end{equation*}
	and once again this is a $*$-homomorphism because of the commutation relations in $C(\bbF)$ and the orthogonality properties of the generators of $C(\bbH)$.
	The universal property of the free product now implies the existence of a unique $*$-homomorphism $\Delta : \mathcal{A} \to C(\bbF)\otimes C(\bbF)$ (recall that $\mathcal{A}$ is the free product appearing in the definition of $C(\bbF)$) which coincides with $\widetilde{\Delta}_{\bbH}$ on $C(\bbH)$ and with $\Phi_{i, \alpha}$ on $C(\bbG_{i, \alpha}))$ so that all that remains to be checked is that $\Delta$ vanishes on elements of the form $[g_{pq}^{i\alpha}, h_{\alpha \beta}]$.
	And indeed,
	\begin{align*}
		\Delta\left(h_{\alpha \beta}g_{pq}^{i\alpha}\right) & = \Delta (h_{\alpha \beta}) \Delta\left(g_{pq}^{(i,\alpha)}\right) \\
		& = \widetilde{\Delta}_\bbH (h_{\alpha \beta})\widetilde{\Delta}_{\bbG_{i,\alpha}}\left(g_{pq}^{(i,\alpha)}\right)\\
		& = \left(\sum_{\gamma \in \Omega} h_{\alpha \gamma} \otimes h_{\gamma \beta}\right) \left(\sum_{\delta\in \Omega_{i}}\sum_{r \in \Lambda_i} h_{\alpha, \delta}g_{pr}^{(i,\alpha)}\otimes g_{rq}^{(i,\delta)}\right) \\
		& = \sum_{\gamma \in \Omega, \delta\in \Omega_{i}}\sum_{r \in \Lambda_i} h_{\alpha \gamma} h_{\alpha, \delta}g_{pr}^{(i,\alpha)}\otimes h_{\gamma\beta}g_{rq}^{(i,\delta)} \\
		& = \sum_{\gamma \in \Omega_i, r \in \Lambda_i} h_{\alpha \gamma}g_{pr}^{(i,\alpha)} \otimes h_{\gamma \beta}g_{rq}^{(i,\gamma)}\\
		& = \sum_{\gamma \in \Omega_i, r \in \Lambda_i} g_{pr}^{(i,\alpha)}h_{\alpha \gamma} \otimes g_{rq}^{(i,\gamma)}h_{\gamma \beta}\\
		& = \Delta\left(g_{pq}^{i\alpha}h_{\alpha \beta}\right).
	\end{align*}
	
	To conclude, we now compute $\Delta(f_{\alpha p, \beta q})$ to check that the formula in the statement holds.
	If there is no $i$ such that $\alpha, \beta \notin \Omega_{i}$, then $f_{\alpha p, \beta q} = 0$, which implies that $\Delta(f_{\alpha p, \beta q}) = 0$.
	In this case, we note that $\sum_{\gamma r \in \Lambda} f_{\alpha p, \gamma r} \otimes f_{\gamma r, \beta q} = 0$, as at least one of the terms in the tensor product is zero in each of the summands.
	If $\alpha, \beta \in \Omega_{i}$, then $f_{\alpha p, \beta q} = h_{\alpha\beta} g_{p,q}^{(i,\alpha)}$.
	Since $\Delta$ is a $*$-homomorphism, we have using the computations above
	\begin{align*}
		\Delta(f_{\alpha p, \beta q}) & = \Delta \left(h_{\alpha \beta} g_{pq}^{(i,\alpha)}\right) \\
		& = \sum_{\gamma \in \Omega_i, r \in \Lambda_i} h_{\alpha \gamma}g_{pr}^{(i,\alpha)} \otimes h_{\gamma \beta}g_{rq}^{(i,\gamma)}\\
		& = \sum_{\gamma \in \Omega_i, r \in \Lambda_i} f_{\alpha p, \gamma r} \otimes f_{\gamma r, \beta q} \\
		& = \sum_{(\gamma, r) \in \Lambda} f_{\alpha p, \gamma r} \otimes f_{\gamma r, \beta q},
	\end{align*}
	where the last line holds because $f_{\alpha p, \gamma r} = 0$  if $\gamma \notin \Omega_{i}$.
	This finishes the proof of the first point.
	
	To prove point $ii.$, we will prove that $f$ is a fundamental representation, in the sense that it is unitary, has unitary conjugate and that its coefficients generate $C(\bbF)$ as a C*-algebra.
	Unitarity follows from a straightforward computation: for $\alpha, \beta\in \Omega_{i}$, using the fact that $g^{(i, \alpha)}$ is unitary we have
	\begin{align*}
		\sum_{(\gamma, r)\in \Lambda} f_{\alpha p, \gamma r}f_{\beta q, \gamma q}^{*} & = \sum_{\gamma\in \Omega_{i}}\sum_{r\in \Lambda_{i}}h_{\alpha\gamma}g_{pr}^{(i, \alpha)}g_{qr}^{(i, \beta)\ast}h_{\beta\gamma} \\
		& = \sum_{\gamma\in \Omega_{i}}\sum_{r\in \Lambda_{i}}h_{\alpha\gamma}h_{\beta\gamma}g_{pr}^{(i, \alpha)}g_{qr}^{(i, \beta)\ast} \\
		& = \delta_{\alpha, \beta}\sum_{r\in \Lambda_{i}}g_{pr}^{(i, \alpha)}g_{qr}^{(i, \alpha)\ast} \\
		& = \delta_{\alpha, \beta}\delta_{p, q}
	\end{align*}
	and the sum vanishes if $\alpha$ and $\beta$ are not in the same orbit.
	Similar computations show that $f$ is unitary, as well as $\overline{f}$.
	As for the generation property, observe that by definition of the free product, $C(\bbF)$ is generated by the set of elements 
	\begin{equation*}
		\{h_{\alpha, \beta}\}_{\alpha, \beta \in \Omega}\sqcup \{g_{pq}^{(i,\alpha)}\}_{i \in [m], \alpha \in \Omega_{i}, p,q \in \Lambda_{i}}.
	\end{equation*}
	It is therefore enough to prove that all these lie in the linear span of $\{f_{\alpha p, \beta q}\}_{\alpha p, \beta q \in \Lambda}$.
	And indeed, we have
	\begin{equation*}
		g_{pq}^{(i, \alpha)} = \sum_{\beta\in \Omega_{i}}f_{\alpha p, \beta q} \qquad \text{and} \qquad h_{\alpha\beta} = \sum_{q\in \Lambda_{i}}f_{\alpha p, \beta q}\left(g_{pq}^{(i,\alpha)}\right)^{*},
	\end{equation*}
	which establishes $ii.$
	
	We now assume that all the fundamental representations are magic unitaries.
	We start by showing that each entry $f_{\alpha p, \beta q}$ is a projection in $C(\bbF)$.
	We may assume that $\alpha, \beta \in \Omega_{i}$, for some $ i \in [m]$.
	Then, 
	\begin{align*}
		f_{\alpha p, \beta q}^{2} & = h_{\alpha \beta } g_{pq}^{(i,\alpha)} h_{\alpha \beta } g_{pq}^{(i,\alpha)} \\
		& =  h_{\alpha \beta } h_{\alpha \beta } g_{pq}^{(i,\alpha)}  g_{pq}^{(i,\alpha)} \\
		& = h_{\alpha \beta } g_{pq}^{(i,\alpha)} \\
		& = f_{\alpha p, \beta q}
	\end{align*}
	since $h_{\alpha \beta },\ g_{pq}^{(i,\alpha)}$ are commuting projections.
	Similarly, one can also show that $f_{\alpha p, \beta q}^{*} = f_{\alpha p, \beta q}$, so that $f_{\alpha p, \beta q}$ is projection for every $(\alpha, p), (\beta, q) \in \Lambda$.
	
	Now, consider $\sum_{(\beta, q) \in \Lambda} f_{\alpha p, \beta q}$ for some fixed $(\alpha, p) \in \Lambda$.
	If $\alpha \in \Omega_{i}$, then we have the following: 
	\begin{align*}
		\sum_{(\beta, q) \in \Lambda} f_{\alpha p, \beta q} & = \sum_{\beta \in \Omega_{i}, q \in \Lambda_{i}} f_{\alpha p, \beta q} \\
		& = \sum_{\beta \in \Omega_{i}, q \in \Lambda_{i}} h_{\alpha,\beta} g_{pq}^{(i, \alpha)} \\
		& = \left(\sum_{\beta \in \Omega_{i}} h_{\alpha,\beta}\left(\sum_{q \in \Lambda_{i}} g_{pq}^{(i, \alpha)}\right)\right) \\
		& = \sum_{\beta \in \Omega_{i}} h_{\alpha,\beta} = 1
	\end{align*}
	Similarly, one can show that for a fixed $(\beta, q) \in \Lambda$, $\sum_{(\alpha, p) \in \Lambda} f_{\alpha p, \beta q} = 1$.
	This shows that $f$ is a magic unitary, thus establishing $iii.$
\end{proof}

We are now ready for the definition of the quantum counterpart of inhomogeneous wreath products.
In the sequel, we will only use it when $\bbG_{1}, \dots, \bbG_{n}$ are quantum permutation groups, hence we state it in that restricted setting.

\begin{definition}
	Let $\bbG_{1}, \dots, \bbG_{m}$ be quantum permutation groups acting on $\Lambda_{1}, \dots, \Lambda_{m}$, respectively, and let $\bbH$ be a quantum permutation group acting on $\Omega$ with orbits $\Omega_{1}, \dots, \Omega_{m}$.
    The \emph{free inhomogeneous wreath product} of $\bbG_{1}, \dots, \bbG_{m}$ with $\bbH$, denoted $(\bbG_{1}, \dots , \bbG_{m}) \fgwr \bbH$, is the compact quantum group $\bbF = (C(\bbF), \Delta)$ acting on $\Lambda = \bigsqcup_{i \in [m]} \Omega_i \times \Lambda_i$, as defined in \autoref{cstr:FIWP} and \autoref{thm:FIWP}.
\end{definition}

Let us briefly explain how this relates to the classical case, the details are direct generalizations of those given in \cite[Sec 7.2.2]{freslon2023compact} for free wreath products.
Assume that $\bbG_1,\ldots,\bbG_m,\bbH$ are all classical, and assume that for each $i \in [m]$ we have a faithful unitary representation $\rho_{i}$ of $\bbG_{i}$ on a finite-dimensional Hilbert space $V_{i}$.
Write
\begin{equation*}
	V = \bigoplus_{i=1}^{m}\bigoplus_{\alpha\in \Omega_{i}}V_{i, \alpha},
\end{equation*}
where $V_{i, \alpha} = V_{i}$.
For each $i\in [m]$ and $\alpha\in \Omega_{i}$, we define a representation $\rho_{i, \alpha}$ of $\bbG_{i, \alpha}$ on $V$ by letting an element $g$ act by $\rho_{i}(g)$ on $V_{i, \alpha}$ and by the identity on all other summands.
Furthermore, denoting by $(e_{p}^{(i, \alpha)})$ an orthonormal basis of $V_{i, \alpha}$, we define a unitary representation of $\bbH$ on $V$ by the formula
\begin{equation*}
	\pi(h)\big(e_{p}^{(i, \alpha)}\big) = e_{p}^{(i, \alpha^{h})}.
\end{equation*}
Then, a simple computation shows that the map sending $(g^{i, \alpha}, h)_{(i, \alpha)\in \Lambda}$ to
\begin{equation*}
	\pi(h)\prod_{i=1}^{m}\prod_{\alpha\in \Omega_{i}}\rho_{i, \alpha}(g^{i, \alpha})
\end{equation*}
defines a faithful unitary representation of the (classical) inhomogeneous wreath product $(\bbG_{1}, \dots, \bbG_{m}) \gwr \bbH$ on $V$.
The coefficients of that representation are
\begin{equation*}
	\rho_{(p\alpha, q\beta)}(g^{i, \alpha}, h) =
	\begin{cases}
		h_{\alpha\beta}\rho\big(g^{(i,\alpha)}\big)_{pq} &\quad \text{if $\alpha,\beta \in \Omega_{i}$, for some $i\in[m]$,}\\
		0 &\quad \text{otherwise.}
	\end{cases}
\end{equation*}
This shows that the free inhomogeneous wreath product is the natural noncommutative analogue of the inhomogeneous wreath product.

It should also be noted that the free inhomogeneous wreath product is a generalization of both the free product and the free wreath product.
Indeed, one can easily prove the following using universal properties: 

\begin{proposition}\label{prop:one-prod-to-rule-them-all}
	\label{prop:free_hom_are_inhom}
	The following hold true for the free inhomogeneous wreath product:
	\begin{enumerate}[label = \roman*.]
		\item If $\bbG_{1} = \dots = \bbG_{m} = \bbG$, then $(\bbG_{1}, \dots , \bbG_{m}) \gwr_{*} \bbH \cong \bbG \wr_{*} \bbH$.
		\item If $\bbH$ is the trivial group, then $(\bbG_{1}, \dots, \bbG_{n}) \gwr_{*} \bbH \cong \bigast_{i=1}^{n} \bbG_{i}$.
	\end{enumerate}
\end{proposition}

\subsection{Application to Disjoint Unions of Graphs}
\label{subsec:disjoint-union}

As a first application of the free inhomogeneous free wreath product, we consider disjoint unions of connected graphs that are either isomorphic or non-quantum isomorphic.
We will show that this disjoint union can be expressed as a free inhomogeneous wreath product of a free product of $\bbS_{N}^{+}$ with the quantum automorphism groups of the individual graphs.
First, we recall some fundamental facts from~\cite[Section 3]{meunier2023quantum} (see also~\cite[Lemma 5.2]{Quantum_Sabidussi}).

\begin{lemma}
	\label{lem:disjoint-union-magic-unitary}
	Let $\{X_i\}_{i \in [n]}$ be coloured connected graphs.
	Let $u$ denote the fundamental representation of the quantum automorphism group of the disjoint union $X = \bigsqcup_{i=1}^n \left(\bigsqcup_{\alpha=1}^{k_i} X_i^\alpha\right)$, where each $X_i^\alpha$ is an isomorphic copy of $X_i$.
	Then, the following hold true: 
	\begin{enumerate}[label = \roman*.]
		\item For any $i \in [n]$, $\alpha \in [k_i]$ and $j \in [n]$, $\beta \in [k_{j}]$, denote the submatrix of $u$ with rows and columns indexed by $V(X_i^\alpha)$ and $V(X_{j}^{\beta})$ by $u^{i\alpha, j\beta}$.
		Then, the rows and columns sum to the same projection, say $e_{i\alpha, j\beta}$ and $A_{X_i}u^{i\alpha,j\beta} = u^{i\alpha,j\beta}A_{X_j}$.
		
		\item The matrix $[e_{i\alpha, j\beta}]$ is a magic unitary, and $[e_{i\alpha, j\beta}, u_{pq}] = 0$ for any $p \in V(X_i^\alpha)$ and $q \in V(X_{j}^{\beta})$.
		
		\item In particular, if $X_i$ and $X_{j}$ are not quantum isomorphic then, $u_{pq} = 0$ for all $p \in V(X_i^\alpha)$, $q \in V(X_{j}^{\beta})$, $\alpha \in [k_i]$ and $\beta \in [k_{j}]$.
		
		\item If $\alpha \in [k_i]$ and $\beta \in [k_j]$ for some $i,j$ such that $X_i$ and $X_j$ are not quantum isomorphic, then $e_{i\alpha, j\beta} = 0$.
	\end{enumerate}
\end{lemma}
\begin{proof}
	\textit{i.} and \textit{ii.} follow from~\cite[Lemma 5.2]{Quantum_Sabidussi}.
	For $\textit{iii.}$, note that if $u_{pq} \neq 0$ for some $p \in V(X_i^\alpha)$, $q \in V(X_{j}^{\beta})$, $\alpha \in [k_i]$ and $\beta \in [k_{j}]$, then it follows from \textit{i.} that $X_i$ and $X_j$ are quantum isomorphic.
	Indeed, the submatrix of $u$ with rows indexed by $V(X_i^\alpha)$ and columns indexed by $V(X_{j}^{\beta})$, with its entries restricted to the range of $e_{i\alpha, j\beta}$ produces a magic unitary intertwining the adjacency matrices $A_{X_i}$ and $A_{X_j}$.
	\textit{iv.} is an immediate consequence of \textit{iii.}
\end{proof}

The quantum automorphism group of a disjoint union of coloured graphs was described in~\cite[Lemma 5.4 and Theorem 5.6]{Q_Aut_Trees} (see also \cite[Thm 5.3]{meunier2023quantum}).

\begin{proposition}
	\label{lem:disjoint-union}
	Let $\{X_i\}_{i \in [n]}$ be pairwise non-quantum isomorphic coloured graphs.
	The quantum automorphism group of the disjoint union $X = \bigsqcup_{i=1}^n \left(\bigsqcup_{\alpha=1}^{k_i} X_i^\alpha\right)$, where each $X_i^\alpha$ is an isomorphic copy of $X_i$, is the given by $\bigast_{i=1}^n \left(\Qut(X_i) \fwr \bbS_{k_i}^+\right)$.
\end{proposition}
%\begin{proof} 
%\end{proof}

We now give an alternate description using the inhomogeneous free wreath product.

\begin{theorem}
	\label{thm:disjoint-union}
	Let $\{X_i\}_{i \in [n]}$ be pairwise non-quantum isomorphic coloured graphs.
	The quantum automorphism group of the disjoint union $X = \bigsqcup_{i=1}^n \left(\bigsqcup_{\alpha=1}^{k_1} X_i^j\right)$, where each $X_i^\alpha$ is an isomorphic copy of $X_i$, is the given by $(\Qut(X_1),\dots,\Qut(X_n)) \fgwr (\bigast_{i=1}^n \bbS_{k_i}^{+})$.
\end{theorem}
\begin{proof}
	
	Let $h$, $g^{i,\alpha}$, $W$, and $w$ denote the fundamental representations of $\bigast_{i=1}^n \bbS_{k_i}^{+}$, the $\alpha$-th copy of $\Qut(X_i)$, $(\Qut(X_1),\dots,\Qut(X_n))\gwr_* (\bigast_{i=1}^n \bbS_{k_i}^{+})$, and $\Qut(X)$ respectively.
	Let us also denote the orbits of $\bigast_{i=1}^n \bbS_{k_i}^{+}$ by $\Omega_i,\dots \Omega_n$, where $\card{\Omega_i} = k_i$ for all $i \in [n]$.

	We begin by showing that $A_XW = WA_X$.
	By the universal property of $\Qut(X)$, this will imply that there is a $\ast$-homomorphism
	\begin{equation*}
	C(\Qut(X)) \to C((\Qut(X_1),\dots,\Qut(X_n))\gwr_* (\bigast_{i=1}^n \bbS_{k_i}^{+}))
	\end{equation*}
	mapping $w_{\alpha p, \beta q}$ to $W_{\alpha p, \beta q}$.
	Note that because $W$ is a fundamental representation, this will necessarily be surjective.

	Therefore, consider $W_{\alpha p, \beta q} W_{\alpha' p', \beta' q'}$, where $\rel(\alpha p, \alpha ' p') \neq \rel(\beta q, \beta'q')$.
	Because $W$ is a magic unitary, it suffices consider the case where $\alpha p \sim \alpha ' p'$, but $\beta q \nsim \beta'q'$.
	If $\alpha, \beta \in \Omega_i$ and $\alpha', \beta' \in \Omega_j$, we have: $W_{\alpha p, \beta q} W_{\alpha' p', \beta' q'} = h_{\alpha \beta} g^{i,\alpha}_{p,q} h_{\alpha'\beta'} g^{j,\alpha'}_{\alpha'\beta'}$.
	Since $\alpha p \sim \alpha ' p'$, we have $\alpha = \alpha'$, and $p \sim p'$.
	Similarly, since $\beta q \nsim \beta'q'$, we have that either $\beta \neq \beta'$, or $q \nsim q'$.
	If $\beta \neq \beta'$, then 
	\begin{align*}
		W_{\alpha p, \beta q} W_{\alpha' p', \beta' q'} & = h_{\alpha \beta} g^{i,\alpha}_{pq} h_{\alpha'\beta'} g^{j,\alpha'}_{p'q'} \\
		& = h_{\alpha \beta} g^{i,\alpha}_{pq} h_{\alpha\beta'} g^{j,\alpha}_{p'q'} \\
		& =  g^{i,\alpha}_{pq} h_{\alpha \beta} h_{\alpha\beta'} g^{j,\alpha}_{p'q'} \\
		& = 0.
	\end{align*}
	On the other hand, if $\beta = \beta'$ and $q \nsim q'$, then $i = j$.
	Hence,   
	\begin{align*}
		W_{\alpha p, \beta q} W_{\alpha' p', \beta' q'} & = h_{\alpha \beta} g^{i,\alpha}_{pq} h_{\alpha'\beta'} g^{j,\alpha'}_{\alpha'\beta'} \\
		& = h_{\alpha \beta} g^{i,\alpha}_{pq} h_{\alpha\beta} g^{i,\alpha}_{\alpha\beta}\\
		& =   h_{\alpha \beta}g^{i,\alpha}_{pq}  g^{i,\alpha}_{p'q'} h_{\alpha\beta} \\
		& = 0,
	\end{align*}
	by the definition of $\Qut(X_i)$.
	
	We now build a surjective $*$-homomorphism in the opposite direction.
	For each $i, \alpha \in \Omega_i$, it follows from~\cref{lem:disjoint-union-magic-unitary} that the submatrix of $w$ indexed by vertices of $X_{i}^{\alpha}$ has rows and columns summing to the same projection $e_{i\alpha, i\alpha}$.
	By the universal property of $\Qut(X_i)$, there is a $*$-homomorphism $\varphi_{i,\alpha}:C(\Qut(X_i)) \to C(\Qut(X))$ mapping $g^{i,\alpha}_{p,q}$ to $w_{\alpha p, \alpha q}$.
	
	Next, consider the magic unitary $e$ from~\cref{lem:disjoint-union-magic-unitary}.
	Since $e_{i\alpha, j\beta} = 0$ for $i \neq j$, by the universal property of $\bigast_{i=1}^n \bbS_{k_i}^{+}$, we see that there exists a $\ast$-homomorphism $\varphi'$ from $C(\bigast_{i=1}^n \bbS_{k_i}^{+})$ to $C(\Qut(X))$ mapping $h_{i\alpha, j\beta}$ to $e_{i\alpha, j\beta}$.
	
	By the universal property of the free product, there is a $\ast$-homomorphism
	\begin{equation*}
	\varphi^{\prime\prime} : \left(\bigast_{i \in [n]}\left(\bigast_{\alpha \in \Omega_i}C(\Qut(X^{i,\alpha}))\right)\right)\ast \left(\bigast_{i=1}^n \bbS_{k_i}^+\right) \to C(\Qut(X))
	\end{equation*}
	extending $\varphi'$ and $\varphi_{i,\alpha}$ for all $i \in [n]$ and $\alpha \in \Omega_i$.
	It is also clear from \cref{lem:disjoint-union-magic-unitary} that for all $i \in [n]$ and $\alpha \in \Omega_i$, $\varphi^{\prime\prime}(f_{i\alpha, j\beta}) = e_{i\alpha, j\beta}$ and $\varphi^{\prime\prime}(g^{i,\alpha}_{pq}) = w_{\alpha p,\beta q}$ commute.
	Hence, $\varphi^{\prime\prime}$ also induces a $\ast$-homomorphism $\varphi$ from $C((\Qut(X_1),\dots,\Qut(X_n))\gwr_* (\bigast_{i=1}^n \bbS_{k_i}^{+}))$ to $C(\Qut(X))$.
	It is also clear that $\varphi$ maps $W_{i\alpha, j\beta}$ to $w_{i\alpha, j\beta}$, which implies that $\varphi$ is surjective.
	This establishes that $\Qut(X)\cong (\Qut(X_1),\dots,\Qut(X_n))\gwr_* (\bigast_{i=1}^n \bbS_{k_i}^{+})$.
\end{proof}

\begin{remark}
	The previous theorem implies that there is a compact quantum group isomorphism
	\[ (\Qut(X_1),\dots,\Qut(X_n))\gwr_* (\bigast_{i=1}^n \bbS_{k_i}^{+}) \cong \bigast_{i=1}^n \left(\Qut(X_i) \wr_* \bbS_{k_i}^+\right). \] 
\end{remark}

\begin{remark}
	\label{rmk:quantum-isomorphic-pair}
	\autoref{thm:disjoint-union} describes the quantum automorphism group of the disjoint union of graphs $X_1,\ldots,X_n$ such that, for all $i,j \in [n]$, the graphs $X_i$ and $X_j$ are either isomorphic or not quantum isomorphic.
	If $X_1$ and $X_2$ are non-isomorphic connected graphs that are quantum isomorphic, then $\Qut(X_1 \sqcup X_2)$ can not be described in this way.
	Indeed, by \cite[Theorem~4.5]{lupini2020nonlocal}, there is an orbit of $\Qut(X_1 \sqcup X_2)$ that intersects both $V(X_1)$ and $V(X_2)$ nontrivially.
	The free inhomogeneous wreath product $(\Qut(X_1),\Qut(X_2)) \gwr_* (\bbS_1^+ * \bbS_1^+) \cong \Qut(X_1) * \Qut(X_2)$ does not have this property, so it is different from $\Qut(X_1 \sqcup X_2)$.
	On the other hand, the free inhomogeneous wreath product $(\Qut(X_1),\Qut(X_2)) \gwr_* \bbS_2^+$ has no well-defined action on $X_1 \sqcup X_2$ since $X_1 \not\cong X_2$, so this also does not describe $\Qut(X_1 \sqcup X_2)$.
	
	It was proved in \cite[Theorem~5.3]{Quantum_Sabidussi} that the quantum automorphism group of an arbitrary disjoint union of graphs can be described in terms of the quantum automorphism groups $\Qut(X_i)$ and the \emph{quantum isomorphism algebras} $\QIso(X_i,X_j)$, as defined in \cite{helton_algebras_nodate}.
\end{remark}

\section{Quantum automorphism groups of connected graphs}
\label{sec:qut_connected}

In this section, we present a general procedure to decompose the quantum automorphism group of a connected graph in terms of certain connected subgraphs, provided that they are either isomorphic or non quantum isomorphic. The key idea is to show that quantum automorphisms preserve the block structure of a connected graph and that they respect ``levels," i.e., showing that quantum automorphisms map blocks on the same level of the block tree to each other. These results will allow us to define an inductive procedure to construct quantum automorphism groups of certain families of graphs in the next section.  

We begin by introducing some notation.
Let $X$ be a connected graph that is not biconnected.
Let $B \in \calB(X)$ be any block other than the central one in $T(X)$.
Then, there is a unique cut vertex $v \in V(B)$, that is a parent of $B$ in the block tree.
We define $\widetilde{B}$, to be the subgraph of $B$ induced by $V(B)\setminus \{v\}$.
If $B_j$ is the central block, we define $\widetilde{B} = B$.
Note that for each vertex $v \in V(X)$, there is a unique $B \in \calB(X)$ such that $v \in V(\widetilde{B})$.
 
As a first result, we observe that since cut vertices and biconnected components are distinguished by the Weisfeiler-Leman algorithm, they are preserved by quantum automorphism. Formally, one has the following (see also~\cite{freslon2025block} for a proof not involving $\wl_X$): 

\begin{lemma}
	\label{lem:2-WL-Blocks}
	Let $X$ be a graph that and $u$ be the fundamental representation of $\Qut(X)$.
	Then, one has the following: 
	\begin{enumerate}[label = \roman*.]
		\item If $p,q \in V(X)$ are in the same block of $X$ and $p',q'\in V(X)$ are not, then $u_{pp'}u_{qq'} = 0$.
		
		\item If $p \in V(X)$ is a cut vertex and $q\in V(X)$ is not, then $u_{pq} = 0$.
	\end{enumerate}
\end{lemma}
\begin{proof}
	It follows from {\cite[Theorem 6]{kiefer2019weisfeiler}} that if $p,q \in V(X)$ are in the same block and $p',q' \in V(X)$ are not then, $\wl_X(p,q) \neq \wl_X(p',q')$.
	Hence, it follows from \cref{lem:wl-quantum-automorphims} that $u_{pp'}u_{qq'} = 0$.
	Similarly, \textit{ii.} follows from \cite[Corollary 7]{kiefer2019weisfeiler} and \cref{lem:wl-quantum-automorphims}.
	See also \cite[Lemma 3.5]{freslon2025block} for an alternative proof.
\end{proof}

Now, we wish to show that blocks at the same ``level" of the block tree are mapped to each other by quantum automorphisms.
We first introduce some terminology, so that we can state this more precisely.
The main point is to decompose the graph using the position of its blocks in the corresponding block tree.
For this, the following notion of \emph{level} will be useful

\begin{definition}
	Let $X$ be a connected graph that is not biconnected.
	We set $T_{0} = T(X)$ and for an integer $i$, we denote by $T_{i+1}$ the tree obtained by removing all the leaves of $T_{i}$.
	Given a vertex $v\in V(T(X))$, there exists a unique integer $i$ such that $v$ is a leaf of $T_{i}$.
	We call it the \emph{level} of $v$ and write $\mathrm{lev}(v) = i$.
\end{definition}

Let us gather some straightforward properties of this notion.

\begin{lemma}
	\label{lem:level_definition}
	Let $X$ be a graph that is not biconnected.
	Then, the following hold true: 
	\begin{enumerate}[label=\roman*.]
		\item If $B \in \calB(X)$ is a block that is a leaf in $T(X)$, then $\lev(B) = 0$.
		\item If $B \in \calB(X)$ and $v \in \calC(X)$, then $\lev(B)$ is even and $\lev(v)$ is odd.
		\item If $v$ is the central node of $T(X)$, then $\lev(v) > \lev(u)$ for all $u \in V(T(X)) \setminus \{v\}$.
	\end{enumerate}
\end{lemma}

Let us introduce two further notations:
\begin{itemize}
	\item $\calB_k(X) \coloneq \{B \in \calB(X) \mid \lev(B) = k\}$, where $k \in \bbN$ is an even number.
	\item $\calC_k(X) \coloneq \{\alpha \in \calC(X) \mid \lev(\alpha) = k \}$, where $k \in \bbN$ is an odd number.
\end{itemize}

The main technical tool in the sequel is the fact that quantum automorphisms respect the level, so that we can work inductively.
Proving this requires some facts relating quantum automorphisms and the levels of blocks and cut-vertices, that we gather in the next lemma.

\begin{lemma}
	\label{lem:levels_are_preserved}
	Let $X$ be a graph and $u$ be the fundamental representation of $\Qut(X)$.
	Then, the following hold true: 
	\begin{enumerate}[label=\roman*.]
		\item Let $L(X) = \cup_{B: \lev(B) = 0} V(\widetilde{B})$. Then, $L(X)$ is a union of orbits of $\Qut(X)$, i.e., $u_{pq}= 0$ if $p \in L(X)$ and $q \notin L(X)$.
		
		\item If $p,q \in V(X)$ are vertices such that $p \in V(\widetilde{B_i})$ and $q \in V(\widetilde{B_j})$, for some blocks $B_i, B_j \in \calB(X)$ with $\lev(B_i) \neq \lev(B_j)$, then $u_{pq} = 0$.
		
		\item Let $B_i$ be a non-central block and let $p \in \calC(X) \cap V(B_i)$ be the unique parent of $B_i$ in $T(X)$.
		Let $B_j \in \calB(X)$, be another block such that $\lev(B_i)=\lev(B_j)$.
		Then $u_{pq} = 0$ for all $q \in V(\widetilde{B_j})$.
		
		\item If $p,q \in \calC(X)$ and $\lev(p) \neq \lev(q)$, then $u_{pq} = 0$.
		
		\item In particular, the center of $T(X)$ is a union of orbits of $\Qut(X)$.
	\end{enumerate}
\end{lemma}
\begin{proof}
	\textit{i.} It follows from \cite[Lemma 8]{kiefer2019weisfeiler} and \cref{lem:wl-quantum-automorphims}.
	Note that \textit{i.} implies that if $p,q \in V(X)$ are vertices such that $p \in V(\widetilde{B_i})$ and $q \in V(\widetilde{B_j})$, for some blocks $B_i, B_j \in \calB(X)$ with $\lev(B_i) = 0$ and $\lev(B_j) > 0$, then $u_{pq} = 0$.
	
	For \textit{ii.}, consider $p \in V(\widetilde{B_i})$ and $q \in V(\widetilde{B_j})$, for some blocks $B_i, B_j \in \calB(X)$ with $\lev(B_i) = 2$ (recall that blocks have even level) and $\lev(B_j) > 2$.
	Let $\widetilde{X}$ denote the graph induced by $V(X)\setminus L(X)$.
	Since $L(X)$ is a union of orbits of $\Qut(X)$, so is $V(X)\setminus L(X)$.
	Hence, it follows from~\cite[Lemma 3.6]{Q_Aut_Trees} that orbits of $\Qut(\widetilde{X})$ are unions of orbits of $\Qut(X)$.
	
	Now, note that the blocks of $\widetilde{X}$ that are leaves in $T(\widetilde{X})$ are precisely the blocks $B$ of $X$ such that $\lev(B) = 2$.
	Hence, it follows from \textit{i.} that $L(\widetilde{X})$ is a union of orbits of $\Qut(X)$.
	In particular, since $p \in L(\widetilde{X})$ and $q \notin L(\widetilde{X})$, we have $u_{pq} = 0$.
	\textit{ii.} now follows by a repeated application of the same argument.
	
	Point \textit{iii.} is an immediate corollary of \textit{ii.} In order to see \textit{iv.}, note that if $\lev(p) = k$ for some cut vertex $p$, other than the central one (if it exists), then $p \in V(\widetilde{B})$, with $\lev(B) = k+1$.
	Hence, \textit{iv.} follows from \textit{ii.} Lastly, \textit{v.} follows directly from \textit{ii.}, \textit{iv.}, and \cref{lem:level_definition}.
\end{proof}

Let us fix yet some more notation before proceeding to the next result.
Given $p \in V(X)$, we define $\calB^p(X) \coloneq \{B_i \in \calB(X): p \in V(B_i)\}$ and $\calB_k^p(X) \coloneq \calB_k(X) \cap \calB^p(X)$.

\begin{lemma}
	\label{lem:gen-res-1}
	Let $X$ be a graph with blocks $\calB(X) = \{B_{i}\}_{i \in [k]}$, cut vertices $\calC(X)$, and fundamental representation $u$.
	Then, the following hold true: 
	\begin{enumerate}[label = \roman*.]
		\item The submatrix of $u$ indexed by the vertices in $\calC(X)$ is a magic unitary.
		
		\item For all $i,j \in [k]$ and all $p \in V(\widetilde{B_{i}})$, $q \in V(\widetilde{B_{j}})$, 
		\[ e_{ij} = \sum_{p'\in V(\widetilde{B_{i}})}u_{p'q} = \sum_{q'\in V(\widetilde{B_{j}})}u_{pq'} \]
		is a well-defined projection independent of the choice of $p$ and $q$.
		In particular, note that if $\lev(B_i)\neq \lev(B_j)$, then $e_{ij} = 0$.
		
		\item Consider the matrix $W$, indexed by $V(T(X))$, defined as follows: 
		\begin{align*}
			W_{xy} = \begin{cases}
				u_{xy} & \text{ if } x,y \in \calC(X),\\
				e_{xy} & \text{ if } x,y \in \calB(X),\\
				0 & \text{ otherwise.}
			\end{cases}
		\end{align*}
		Then, $W$ is a magic unitary that commutes with $A_{T(X)}$.
		
		\item Let $B_i, B_j$ be two non-central blocks, and let $\alpha \in V(B_i)$ and $\beta \in V(B_j)$ be the unique parents of $B_i$ and $B_j$ in $T(X)$.
		Then, one has 
		\begin{align*}
			W_{\alpha\beta} & = u_{\alpha\beta} = \sum_{B_{j'}: \beta \in V(B_{j'})\setminus V(\widetilde{B_{j'}}) } e_{ij'} = \sum_{B_{i'}: \alpha \in V(B_{i'})\setminus V(\widetilde{B_{i'}}) } e_{i'j}  
		\end{align*}
	\end{enumerate}
\end{lemma}
\begin{proof}
	\begin{enumerate}[label = \textit{\roman*.}]
		\item It follows from~\cref{lem:2-WL-Blocks} that if $p \in \calC(X)$, then $u_{pq'}  = 0$ as soon as $q' \notin \calC(X)$.
		This means that when extracting the submatrix of $u$ with rows and columns indexed by vertices in $\calC(X)$, all the coefficients that we remove in these rows and columns are $0$, so that we still have a magic unitary matrix.
		
		\item {Let $p \in V(\widetilde{B_{i}})$, $q \in V(\widetilde{B_{j}})$ be two vertices.
		It is clear from \cref{lem:levels_are_preserved}~\textit{ii.} that if $\lev(B_i) \neq \lev(B_j)$, then $e_{ij}$ is well-defined and vanishes.
		Hence, we assume $\lev(B_i) = \lev(B_j)\coloneq k$ and consider $\sum_{p' \in V(\widetilde{B_{i}})}u_{p'q}$.
		\begin{align*}
			\sum_{p' \in V(\widetilde{B_{i}})} u_{p'q} & = \left(\sum_{p' \in V(\widetilde{B_{i}})} u_{p'q}\right)\left(\sum_{q'\in V(X)} u_{pq'}\right)\\
			& = \left(\sum_{p' \in V(\widetilde{B_{i}})} u_{p'q}\right)\left(\sum_{B_{j'} \in \calB_k(X)}\left(\sum_{q' \in V(\widetilde{B_{j'}})} u_{pq'}\right)\right)\\
			& = \left(\sum_{p' \in V(\widetilde{B_{i}})} u_{p'q}\right)\left(\sum_{B_{j'} \in \calB^q_k(X)}\left(\sum_{q' \in V(\widetilde{B_{j'}})} u_{pq'}\right)\right)\\
			& = \left(\sum_{p' \in V(\widetilde{B_i})} u_{p'q}\right)\left(\sum_{q'\in V(\widetilde{B_{j}})} u_{pq'}\right),
		\end{align*}
		where the second equality follows from \cref{lem:levels_are_preserved}~\textit{ii.} and the fact that $q \in \widetilde{B_j}$ and $\lev(B_j) = k$.
		The third equality follows from~\cref{lem:2-WL-Blocks} and~\cref{lem:levels_are_preserved}~\textit{iii.}
		Similarly, the last equality follows since for each vertex $r \in V(X)$, there is at most one block $B_l \in \calB_k(X)$, such that $r \in V(B_l)$.
		
		Proceeding in a similar manner, we have
		\begin{align*}
		\sum_{q'\in V(\widetilde{B_{j}})} u_{pq'}& = \left(\sum_{p' \in V(X)} u_{p'q}\right)\left(\sum_{q'\in V(\widetilde{B_{j}})} u_{pq'}\right) \\
		& =  \left(\sum_{B_{i'} \in \calB_k(X)}\left(\sum_{p' \in V(\widetilde{B_{i'}})} u_{p'q}\right)\right)\left(\sum_{q'\in V(\widetilde{B_{j}})} u_{pq'}\right) \\
		& =  \left(\sum_{B_{i'} \in \calB^p_k(X)}\left(\sum_{p' \in V(\widetilde{B_{i'}})} u_{p'q}\right)\right)\left(\sum_{q'\in V(\widetilde{B_{j}})} u_{pq'}\right) \\
		& = \left(\sum_{p' \in V(\widetilde{B_i})} u_{p'q}\right)\left(\sum_{q'\in V(\widetilde{B_{j}})} u_{pq'}\right) \\
		\end{align*}
		Hence, $e_{ij}$ is indeed well-defined and independent of the choice of $p$ and $q$.
		
		\item It follows from \textit{i.} and \textit{ii.} that $W$ is a magic unitary.
		It is also clear from the definition of $W$ that if $\alpha \in \calC(X)$, and $i \in \calB(X)$, then $W_{\alpha i}= W_{i\alpha} = 0$.
		Therefore, by Lemma \ref{lem:commutation_adjacency}, we only have to check that for all $i,j \in \calB(X)$ and $\alpha,\beta \in \calC(X)$, such that $\alpha \in V(B_i)$ and $\beta \notin V(B_j)$, one has $W_{ij} W_{\alpha\beta} = 0$.
		We may assume that $\lev(\alpha) = \lev(\beta)$ and $\lev(B_i) = \lev(B_j)$.
		Choose some $p \in V(\widetilde{B_i})$ and write
		\begin{align*}
			W_{ij} W_{\alpha\beta} & = \left(\sum_{q' \in V(\widetilde{B_j})}u_{pq'}\right) u_{\alpha\beta} = \left(\sum_{q' \in V(\widetilde{B_j})}u_{pq'}u_{\alpha\beta}\right).
		\end{align*}
		Since $p$ and $\alpha$ are in the same block, we can conclude by~\cref{lem:2-WL-Blocks} if we show that $q'$ and $\beta$ cannot be in the same block.
		It is of course enough to consider the case where $q'$ is a cut-vertex, and because its parent block is $B_{j}$, $\beta$ must belong to one of its children blocks.
		But then, $\lev(\beta) = \lev(q') + 2 = \lev(B_{j}) + 3 \neq \lev(\alpha)$, a contradiction since we assumed $\lev(\beta) = \lev(\alpha)$.}
		
		\item Consider $u_{\alpha\beta}$, where $\alpha\in V(B_i)$ and $\beta \in V(B_j)$ are defined as in \textit{iv.}
		Note that $\lev(\alpha) = \lev(B_{i'})+1$ and $\lev(\beta) = \lev(B_{j'})+1$, for any blocks $B_{i'}, B_{j'}$ such that $\alpha \in V(B_{i'})\setminus V(\widetilde{B_{i'}})$ and $\beta \in V(B_{j'})\setminus V(\widetilde{B_{j'}})$.
		Hence, it follows from \textit{ii.} and~\cref{lem:levels_are_preserved} that if $\lev(\alpha)\neq \lev(\beta)$, then 
		\begin{align*}
			W_{\alpha\beta} & = u_{\alpha\beta} = \sum_{B_{j'}: \beta \in V(B_{j'})\setminus V(\widetilde{B_{j'}}) } e_{ij'} = \sum_{B_{i'}: \alpha \in V(B_{i'})\setminus V(\widetilde{B_{i'}}) } e_{i'j} = 0.
		\end{align*}
		
		Hence, we may assume that $\lev(\alpha) = \lev(\beta) \coloneq k$ and $\lev(B_i) = \lev(B_j) = k-1$.
		Choose some $p'\in V(\widetilde{B_i})$.
		Then, one has 
		\begin{align*}
			u_{\alpha\beta} & = u_{\alpha\beta}\left(\sum_{q'\in V(X)} u_{p'q'}\right)\\
			& = u_{\alpha\beta} \left(\sum_{B_{j'} \in\calB_{k-1}(X)} \left(\sum_{q'\in V(\widetilde{B_{j'}})} u_{p'q'}\right)\right)\\
			& = u_{\alpha\beta} \left(\sum_{B_{j'} \in \calB_{k-1}^{\beta}(X)} \left(\sum_{V(\widetilde{B_{j'}})} u_{p'q'}\right)\right),
		\end{align*}
		where the second equality follows from \cref{lem:levels_are_preserved} and the third one follows from \cref{lem:2-WL-Blocks}.
		In the above expression, note that 
		\[ \left(\sum_{B_{j'} \in \calB_{k-1}^{\beta}(X)} \left(\sum_{V(\widetilde{B_{j'}})} u_{p'q'}\right)\right) = \sum_{B_{j'}: \beta \in V(B_{j'})\setminus V(\widetilde{B_{j'}}) } e_{ij'}. \]
		Next, we have 
		\begin{align*}
			\left(\sum_{B_{j'} \in \calB_{k-1}^{\beta}(X)} \left(\sum_{V(\widetilde{B_{j'}})} u_{p'q'}\right)\right) & = \left(\sum_{\alpha'\in V(X)}u_{\alpha'\beta}\right)\left(\sum_{B_{j'} \in \calB^\beta_{k-1}(X)} \left(\sum_{V(\widetilde{B_{j'}})} u_{p'q'}\right)\right)\\
			& = \left(\sum_{\alpha'\in \calC_k(X)}u_{\alpha'\beta}\right)\left(\sum_{B_{j'} \in \calB_{k-1}^{\beta}(X)} \left(\sum_{V(\widetilde{B_{j'}})} u_{p'q'}\right)\right),\\
		\end{align*}
		by \cref{lem:2-WL-Blocks} and \cref{lem:levels_are_preserved}.
		To conclude, note that it follows from \cref{lem:2-WL-Blocks} that in the above sum, $u_{\alpha'\beta}u_{p'q'} = 0$, unless $\alpha' \in V(B_i)$.
		However, $\alpha$ is the unique vertex with $\lev(\alpha)=k$ such that $\alpha\in V(B_i)$.
		Hence, one has 
		\[ \left(\sum_{B_{j'} \in \calB_{k-1}^{\beta}(X)} \left(\sum_{V(\widetilde{B_{j'}})} u_{p'q'}\right)\right) = u_{\alpha\beta} \left(\sum_{B_{j'} \in \calB_{k-1}^{\beta}(X)} \left(\sum_{V(\widetilde{B_{j'}})} u_{p'q'}\right)\right). \]
		This proves the first equality, and the second one follows from similar arguments.
	\end{enumerate}
\end{proof}

\begin{remark}
The statement of~\Cref{lem:levels_are_preserved} implies that the fundamental representation $u$ of $X$ is an isomorphism of partitioned graphs in the sense of~\cite{freslon2025block}, where the partition is given by the levels. This enables to recover the first two points of~\Cref{lem:gen-res-1} from~\cite[Theorem 3.3]{freslon2025block}.
\end{remark}

We shall also need the following lemma for the proof the next theorem.
Recall that we defined $X^{\leq \alpha}$ (in \cref{subsec:prelim:graphs}) to be the subgraph of $X$ induced by all blocks and cut vertices below $\alpha$ in the block tree, including $\alpha$ itself.

\begin{lemma}
	\label{lem:colouring-block-by-iso-type}
	Let $B$ be a block of a connected graph $X$, and $u$ be the fundamental representation of $X$.
	If $\alpha, \beta \in \calC(X)$ are two cut vertices in $B$, then
	\begin{enumerate}[label = \roman*.]
		\item For any $p \in V(X^{\leq \alpha})$ and $q \in V(X^{\leq \beta})$, the following holds true: 
		\[ u_{\alpha\beta} = \sum_{q' \in V(X^{\leq \beta})} u_{pq'} = \sum_{p'\in V(X^{\leq \alpha})} u_{p'q}. \]
		
		\item In particular, $u_{\alpha\beta} = 0$ if and only if $u_{pq} = 0$ for all $p \in V(X^{\leq \alpha})$ and $q \in V(X^{\leq \beta})$.
		
		\item Let $u^{\alpha\beta}$ denote the submatrix of $u$ with rows indexed by $V(X^{\leq \alpha})$ and columns indexed by $V(X^{\leq \beta})$.
		Then, the rows and columns of $u^{\alpha\beta}$ sum to the same projection and $A_{X^{\leq\alpha}}u^{\alpha\beta} = u^{\alpha\beta}A_{X^{\leq\beta}}$.
		Hence, $X^{\leq \alpha} \not\cong_{q} X^{\leq \beta}$ implies $u_{\alpha\beta} = 0$.
	\end{enumerate}
\end{lemma}
\begin{proof}
	Let $\lev(\alpha) = \lev(\beta) \coloneq k$.
	Choose some $p \in V(X^{\leq \alpha})$ such that $p \in \widetilde{B_i}$, for some $B_i \in \calB(X)$.
	We proceed by induction on $\lev(B_i)$.
	
	First consider the case where $\lev(B_i) = k+1$.
	In this case, we have $p = \alpha$.
	Moreover, 
	\begin{align*}
		\sum_{q'\in V(X^{\leq \beta})}u_{\alpha q'} = \left(\sum_{\substack{B_{j'} \in\calB_{k+1}(X) \\ \& \ B_{j'}\in T(X^{\leq \beta})}} \left(\sum_{q'\in V(\widetilde{B_{j'}})} u_{\alpha q'}\right)\right) = u_{\alpha\beta},
	\end{align*}
	since $\beta$ is clearly the only vertex $q'$ in $V(X^{\leq \beta})$ such that $q'\in V(\widetilde{B_{j'}})$ for some $B_{j'} \in\calB_{k+1}(X)$.
	
	Assume now that we have shown that the result holds for all $p \in V(X^{\leq \alpha})$ such that $p \in V(\widetilde{B_{i'}})$ and $B_{i'} \in \calB_{k- 2l+1}(X)$.
	Choose some $p \in V(X^{\leq \alpha})$ such that $p \in V(\widetilde{B_i})$ for some $ B_i \in \calB_{k- 2l+3}(X)$.
	Now, note that 
	\[ \sum_{q'\in V(X^{\leq \beta})}u_{pq'} = \left(\sum_{\substack{B_{j'} \in\calB_{k-2l+3}(X) \\ \& \ B_{j'}\in T(X^{\leq \beta})}} \left(\sum_{q'\in V(\widetilde{B_{j'}})} u_{p q'}\right)\right). \]
	
	Let $p_i \in V(B_i) \cap \calC(X)$, be the unique parent of $B_i$ in $T(X)$ and choose a fixed $q_k \in \calC_{k-2l+2}(X)$.
	Then, it follows from \cref{lem:gen-res-1}~\textit{iv.} that 
	\[ u_{p_{i}q_{k}} = \left(\sum_{B_{j'} \in \calB^{q_k}_{k-2l+3}(X)}\left(\sum_{q'\in V(\widetilde{B_{j'}})} u_{p q'}\right)\right). \]
	Since each block $B_{j'}\in \calB_{k-2l+1}(X^{\leq \beta})$ is the child of a unique $q_{k'} \in \calC_{k-2l+2}(X^{\leq\beta})$ in $T(X)$, we may write
	\begin{align*}
		\sum_{q'\in V(X^{\leq \beta})}u_{pq'} & = \left(\sum_{\substack{B_{j'} \in\calB_{k-2l+3}(X) \\ \& \ B_{j'}\in T(X^{\leq \beta})}} \left(\sum_{q'\in V(\widetilde{B_{j'}})} u_{p q'}\right)\right)\\
		& = \left(\sum_{q_{k'} \in \calC_{k-2l+2}(X)}\left(\sum_{\substack{B_{j'} \in \calB^{q_k'}_{k-2l+3}(X)\\ \& B_{j'}\in T(X^{\leq \beta})}}\left(\sum_{q'\in V(\widetilde{B_{j'}})} u_{p q'}\right)\right)\right)\\
		& = \sum_{q_{k'} \in \calC_{k-2l+2}(X^{\leq \beta})} u_{pq_{k'}} = u_{\alpha\beta}, 
	\end{align*}
	where the final equality follows from the induction hypothesis.
	
	It is now clear from \textit{i.} that $u_{\alpha\beta} = 0$ if and only if $u_{pq} = 0$ for all $p \in V(X^{\leq \alpha})$ and $q \in V(X^{\leq \beta})$.
	Lastly, note that if $u_{\alpha\beta}\neq 0$, then the submatrix of $u$ with rows indexed by $V(X^{\leq \alpha})$ and columns indexed by $V(X^{\leq \beta})$, is a quantum permutation matrix when restricted to the range of $u_{\alpha\beta}$.
	Let us denote this matrix by $u^{\alpha\beta}$.
	It is also not too difficult to see that $A_{X^{\leq \alpha}} u^{\alpha\beta } = u^{\alpha\beta} A_{X^{\leq \beta}}$.
	Hence, if $u_{\alpha\beta} = 0$, we have $X^{\leq \alpha} \cong_q X^{\leq \beta}$, which establishes the contrapositive of the second part of \textit{iii.}
\end{proof}

We now prove the main theorem of this section. 

\begin{theorem}
	\label{thm:block_case}
	Let $B$ be a block of a connected graph $X$.
	Suppose that two children cut vertices $\alpha, \beta$ of $B$ have the same colour if and only if $X^{\leq \alpha}\cong X^{\leq \beta}$.
	Further, assume that $X^{\leq \alpha}_\alpha \cong X^{\leq \beta}_\beta$ if and only if $X^{\leq \alpha}_\alpha \cong_q X^{\leq \beta}_\beta$, for all children cut vertices $\alpha,\beta$ of $B$.
	\begin{enumerate}[label = \roman*.,itemsep=0cm]
		\item
		If $B$ is the central block and $\alpha_1,\dots, \alpha_n$ are representatives of orbits of $\Qut_c(B)$ acting on $V(B)\cap\calC(X)$, then
		\[ \Qut(X) \cong (\Qut(X^{\leq \alpha_1})_{\alpha_1},\dots,\Qut(X^{\leq \alpha_n})_{\alpha_n},1,\dots,1)\fgwr\Qut_c(B), \]
		where the orbits of $\Qut_c(B)$ are $\Omega_1,\dots,\Omega_{n}, \Omega_{n+1},\dots,\Omega_m$, where $\bigcup_{i=1}^n \Omega_i = \calC(X)\cap V(B)$.
		
		\item
		If $B$ is not a central block with a parent cut vertex $\alpha$ and $\alpha_1,\dots, \alpha_n$ are representatives of orbits of $\Qut_c(B)_\alpha$ acting on $V(B)\cap\calC(X)$, then
		\[ \Qut(X^{\leq B})_\alpha \cong (\Qut(X^{\leq \alpha_1})_{\alpha_1},\dots,\Qut(X^{\leq \alpha_n})_{\alpha_n},1, \dots, 1)\fgwr\Qut_c(B)_\alpha, \]
		where the orbits of $\Qut_c(B)_v$ are $\Omega_1,\dots,\Omega_{n}, \Omega_{n+1},\dots,\Omega_m$, where $\bigcup_{i=1}^n \Omega_i = \calC(X)\cap V(B)$.
	\end{enumerate}
\end{theorem}
\begin{proof}
	We shall only prove \textit{i.}, as the proof of \textit{ii.} can be done similarly.
	Let us denote the quantum group $(\Qut(X^{\leq \alpha_1})_{\alpha_1},\dots,\Qut(X^{\leq \alpha_n})_{\alpha_n},1,\dots,1)\fgwr\Qut_c(B)$ by $F$, and let us denote the fundamental representations of $\Qut(X)$ and $F$ by $u$ and $w$ respectively.
	Let us also denote the fundamental representations of $\Qut_c(B)$ by $h$ and  the fundamental representation of the $k$-th copy of $\Qut(X^{\leq\alpha_i})_{\alpha_i}$ by $g^{(i,k)}$ for $i= 1,\dots, n$ and $k \in \Omega_i$.
	
	Before we proceed, we show that $F$ is indeed a quantum permutation group acting on $V(X)$.
	In order to see this, we first note that the set that $F$ acts is given by
	\[ \left(\bigcup_{i \in [n]} \Omega_i \times V(X^{\leq \alpha_i})\right) \cup \left(\bigcup_{i = n+1}^m \Omega_i\right). \]
	Next, for a fixed $i$ and any $\alpha_i^k, \alpha_i^l \in \Omega_i$, one has $X^{\leq \alpha_i^k} \cong X^{\leq \alpha_i^l} \cong X^{\leq \alpha_i}$.
	Let $\phi_i^k: V(X^{\leq\alpha_i^k}) \to V(X^{\leq \alpha_i})$, be any isomorphism between $X^{\leq \alpha_i^k}$ and $X^{\leq \alpha_i}$ mapping $\alpha_i^k$ to $\alpha_i$.
	We have a bijection
	\[ \phi : V(X) \to \left(\bigcup_{i \in [n]} \Omega_i \times V(X^{\leq \alpha_i})\right) \cup \left(\bigcup_{i = n+1}^m \Omega_i\right)	\]
	defined as follows: 
	\begin{align*}
		\phi(v) & = \begin{cases}
			v & \text{ if } v \in V(B) \setminus \calC(X) \\
			(\alpha_i^k, \phi_i^k(v)) & \text{ if } v \in V(X^{\leq \alpha_i^k}) \text{ for some } i \in [n] \text{ and } k \in \Omega_i.
		\end{cases}
	\end{align*}
	Hence, we will assume from now on that $w$ is indexed by $V(X)$ without any ambiguity.
	
	We begin by checking that $w_{pq}w_{p'q'} = 0$ if $p \sim p'$ and $q \nsim q'$.
	By Lemma \ref{lem:commutation_adjacency}, this will establish the existence of a $\ast$-homomorphism $\psi$ from $C(\Qut(X))$ to $C(F)$ mapping $u_{pq}$ to $w_{pq}$, for all $p,q \in V(X)$.
	We consider the following cases:
	\begin{enumerate}[label=Case \Roman* -]
		\item $p,q, p',q' \in V(B)$: We claim that in that case, $w_{pq} = h_{pq}$.
		Indeed, if $p\in V(B)\setminus\calC(X)$ we may assume $q\in V(B)\setminus\calC(X)$, since otherwise we have $w_{pq} = h_{pq} = 0$.
		Then, $w_{pq} = h_{pq}$ by definition.
		If on the contrary $p, q\in \calC(X)$, then either they are in different orbits, in which case $w_{pq} = 0 = h_{pq}$, or there exists $i\in [n]$ such that $p, q\in \Omega_{i}$.
		This means that $p = \alpha_{i}^{k}$ and $q = \alpha_{i}^{l}$ and therefore $w_{pq} = h_{\alpha_{i}^{k}\alpha_{i}^{l}}g^{i, k}_{\alpha_{i}\alpha_{i}} = h_{pq}$ since $g$ fixes $\alpha_{i}$.
		The result now follows since $w_{pq}w_{p'q'} = h_{pq}h_{p'q'}$ vanishes if $p\sim p'$ and $q\nsim q'$ or if $p\nsim q$ and $p'\sim q'$.
		
		\item $p,p' \in V(B)$ and $q\in V(X^{\leq \alpha_i^k})$ and $q'\in V(X^{\leq \alpha_j^l})$, for some $i, j \in [n]$, $k \in \Omega_i$, and $l\in \Omega_j$.
		Thanks to the previous case, we may assume $q,q'\notin V(B)$.
		We claim that in that case $w_{pq} = 0$ and there is nothing to prove.
		Indeed, by case I,
		\begin{equation*}
			\sum_{r \in V(B)} w_{pr} = \sum_{r \in V(B)}h_{pr} = 1
		\end{equation*}
		hence $w_{pq} = 0$ for all $q \notin V(B)$.
		
		\item $p,p' \in V(X^{\leq \alpha_i^k})$, for some $i \in [n]$ and $k \in \Omega_i$, and $q,q'\in V(B)$: This is the same as the previous case and it can be dealt with similarly.
		
		\item $p,p' \in V(X^{\leq \alpha_i^k})$, $q\in V(X^{\leq \alpha_j^{l}})$ and $q'\in V(X^{\leq \alpha_{j'}^{l'}})$, for some $i, j, j' \in [n]$, $k \in \Omega_i$, $l \in \Omega_j$ and $l'\in \Omega_{j'}$: We may assume that $i = j = j'$ since otherwise at least one of $w_{pq}$ or $w_{p'q'}$ vanishes.
		Hence, we have $w_{pq} = h_{\alpha^i_k\alpha^i_l}g^{(i,k)}_{p q}$ and $w_{p'q'} = h_{\alpha^i_k\alpha^i_{l'}}g^{(i,k)}_{p'q'}$.
		Thus,
		\begin{align*}
			w_{pq}w_{p'q'} & = h_{\alpha^i_k\alpha^i_l}g^{(i,k)}_{p, q}h_{\alpha^i_k\alpha^i_{l'}}g^{(i,k)}_{p'q'} \\
			& = h_{\alpha^i_k\alpha^i_l}g^{(i,k)}_{p, q}g^{(i,k)}_{p'q'}h_{\alpha^i_k\alpha^i_{l'}} = 0,
		\end{align*}
		since $p\sim p'$ and $q\nsim q'$.
	\end{enumerate}
	
	We shall now construct a surjective $\ast$-homomorphism in the opposite direction, i.e., from $C(F)$ to $C(\Qut(X))$, mapping $u_{pq}$ to $w_{pq}$.
	It follows from \cref{lem:levels_are_preserved}~\textit{v.} that the central block $V(B)$ is a union of orbits, so that the submatrix of $u$ indexed by $V(B)$ is a magic unitary satisfying $u_{\alpha\beta}u_{\alpha'\beta'} = 0$ whenever $\rel(\alpha,\alpha') \neq \rel(\beta, \beta')$.
	Moreover, it follows from \cref{lem:colouring-block-by-iso-type} that $u_{\alpha\beta} = 0$ if $X^{\leq \alpha}\not \cong_{qc} X^{\leq\beta}$.
	By our assumption that $X^{\leq \alpha} \cong_{qc} X^{\leq \beta}$ if and only if $X^{\leq \alpha} \cong X^{\leq \beta}$, we can conclude that $u_{\alpha\beta} = 0$ if $X^{\leq \alpha}\not \cong X^{\leq\beta}$.
	Therefore, the submatrix of $u$ also respects the colouring $c$.
	Hence, we have a $\ast$-homomorphism  $\varphi'$ from $\Qut_c(B)$ to $\Qut(X)$ mapping $h_{\alpha\beta}$ to $u_{\alpha\beta}$.
	
	Similarly, it follows from \cref{lem:colouring-block-by-iso-type} that for each $i \in [n]$ and $k \in \Omega_i$, the submatrix of $u$ indexed by $V(X^{\leq \alpha_i^k})$ commutes with $A_{X^{\leq \alpha_i^k}}$ and has rows and columns summing to the same projection $u_{\alpha_i^k, \alpha_i^k}$.
	Hence, for each $i \in [n]$ and $k \in \Omega_i$, we have a $\ast$-homomorphism $\varphi^{(i,k)}: C(X^{\leq\alpha_i^k}) \to C(\Qut(X))$ mapping $g^{(i,k)}_{pq}$ to $u_{pq}$.
	By the universal property of the free product, we may conclude that there is a $\ast$-homomorphism
	\[ \varphi^{\prime\prime}: \left(\bigast_{i=1}^n \left( \bigast_{k \in \Omega_i} C(\Qut(X^{\leq \alpha_i^k}))_v \right)\right) \ast \Qut_c(B)\to C(\Qut(X)) \]
	extending $\varphi'$ and $\varphi^{(i,k)}$ for all $i \in [n]$ and $k \in \Omega_i$.
	
	Lastly, we check that $\varphi^{\prime\prime}(h_{\alpha_i^k\beta}) = u_{\alpha_i^k\beta}$ commutes with $\varphi^{\prime\prime}(g^{(i,k)}_{pq}) = u_{pq}$ if $\alpha_i^k \in \Omega_i$.
	We may assume that $u_{\alpha_i^k\beta} \neq 0$ and $u_{pq} \neq 0$, so that $\beta = \alpha_{i}^{l}\in \Omega_i$.
	It follows from \cref{lem:colouring-block-by-iso-type} that for $q\in X^{\leq \alpha_{i}^{l}}$,
	\[ u_{\alpha_i^k\alpha_i^l} = \sum_{p'\in V(X^{\leq \alpha_i^k })}u_{p'q}. \]
	It is now easy to see that $u_{\alpha\beta}$ and $u_{pq}$ commute if $q \in V(X^{\leq \beta})$.
	Hence, $\varphi^{\prime\prime}$ induces a $\ast$-homomorphism $\varphi: C(F) \to C(\Qut(X))$ mapping the corresponding entries of $w$ to those of $u$, which shows that $\varphi$ is surjective.
	One easily checks on the generators that $\varphi$ and $\psi$ are mutual inverses, concluding the proof.
\end{proof}

We shall also need the following lemma to prove the next result: 
\begin{lemma}
	\label{lem:neighbours-of-fixed-vertex}
	Let $v$ be a vertex in a coloured graph $X$ such that $v$ is fixed by $\Qut(X)$.
	Then, the neighbours of $v$ form a union of orbits of $\Qut(X)$.
\end{lemma}
\begin{proof}
	Let $w \in V(X)$ be a neighbour of $v$, and $w'$ be a vertex at distance at least $2$ from $v$.
	We have $u_{ww'} = u_{vv}u_{ww'} = 0$, since $u_{vv} = 1$.
	However, since $v\sim w$ and $v\nsim w'$, we have that $u_{vv}u_{ww'} = 0$.
	Hence, the neighbours of $v$ form a union of orbits of $\Qut(X)$.
\end{proof}

The next two results deal with the case of a central/ non-central cut vertex in the block tree.
\begin{theorem}
	\label{thm:non-central-cut}
	Let $X$ be a connected graph and $\alpha$ be a cut vertex $\alpha \in \calC(X)$.
	Suppose that $B_1,\dots, B_n$ are children blocks of $\alpha$ such that for each child block $B$ of $\alpha$ we have $X^{\leq B}\cong X^{B_i}$, for some $i$, and the isomorphism class $X^{\leq B_i}$ is of size $k_i$.
	Further assume that $X^{\leq B_i} \cong_{q} X^{\leq B_j}$ if and only if $X^{\leq B_i} \cong X^{\leq B_j}$.
	Then, we have $$\Qut(X^{\leq \alpha})_\alpha \cong \bigast_{i=1}^n\Qut(X^{\leq B_i})_\alpha \fwr \bbS_{k_i}^+.$$
\end{theorem}
\begin{proof}
	By~\cref{lem:neighbours-of-fixed-vertex}, we see that the neighbours of $\{\alpha\}$ form a union of orbits of $\Qut(X^{\leq \alpha})_{\alpha}$.
	Hence, it follows from~~\cite[Lemma 4.1~(III)]{Q_Aut_Trees} that we may colour $\alpha$ and its neighbours with distinct colours $c_1$ and $c_2$ not appearing elsewhere in $X^{\leq\alpha}$ without changing $\Qut(X^{\leq\alpha})_\alpha$.
	Moreover, by~\cite[Lemma 4.1~(II)]{Q_Aut_Trees}, we may delete the edges between $\alpha$ and its neighbours still without changing $\Qut(X^{\leq\alpha})_\alpha$.
	Let the graph formed by deleting the edges between $\alpha$ and its neighbours be $\widetilde{X}$.
	Since $\alpha$ was a cut vertex, $\widetilde{X}$ is disconnected and its connected components are precisely the subgraphs of $X^{\leq B_i}$ induced by $V(X^{\leq B_i})\setminus\{\alpha\}$, alongside $\{ \alpha\}$.
	
	For a fixed $i$, let us denote this graph by $X_1^{\leq B_i}$.
	It follows from~\cite[Lemma 4.1~(IV)]{Q_Aut_Trees} that $\Qut(\widetilde{X}^{\leq B_i}) = \Qut(X_1^{\leq B_i}\sqcup\{\alpha\})$, where $\alpha$ is coloured with $c_1$.
	Since the vertices which were neighbours of $\alpha$ in $X$ are coloured with $c_2$ in $\widetilde{X}$, it follows from~\cite[lemma 4.1~(II)]{Q_Aut_Trees} that we may join $\alpha$ to each of its former neighbours which are in $X^{\leq B_i}$ with an edge.
	This graph, which we denote by $X_2^{\leq B_i}$, is $X^{\leq B_i}$ with $\alpha$ coloured in $c_1$ and its neighbours coloured in $c_2$.
	
	Because $\alpha$ is fixed, it follows from~\cref{lem:neighbours-of-fixed-vertex} that its neighbours form a union of orbits.
	Therefore, by another application of~\cite[Lemma 4.1~(III)]{Q_Aut_Trees} that we may de-colour the neighbours of $\alpha$ in $X_2^{\leq B_i}$ without changing its quantum automorphism group.
	Let us denote this graph by $X_3^{\leq B_i}$.
	Since $\alpha$ is the only vertex coloured $c_1$, it is not difficult to see that $\Qut(X_3^{\leq B_i}) \cong \Qut(X^{\leq B_i})_\alpha$.
	Hence, we have $\Qut(X_1^{\leq B_i})\cong \Qut(X^{\leq B_i})_{\alpha}$.
	Since this is true for an arbitrary $i$, the result now follows from a direct application of~\cref{lem:disjoint-union}.
\end{proof}

The following is a direct corollary of~\cref{thm:non-central-cut}. Note that it was also independently obtained by other means in~\cite{FMPbis2025} (see the end of Section 5).

\begin{corollary}
	\label{thm:central-cut}
	Let $X$ be a connected graph such that the central vertex of the block tree is a cut vertex $\alpha \in \calC(X)$.
	Suppose that $B_1,\dots, B_n$ are children blocks of $\alpha$ such that for each child block $B$ of $\alpha$ we have $X^{\leq B}\cong X^{B_i}$, for some $i$, and the isomorphism class $X^{\leq B_i}$ is of size $k_i$.
	Further assume that $X^{\leq B_i} \cong_{q} X^{\leq B_j}$ if and only if $X^{\leq B_i} \cong X^{\leq B_j}$.
	Then, we have $$\Qut(X) \cong \bigast_{i=1}^n\Qut(X^{\leq B_i})_\alpha \fwr \bbS_{k_i}^+.$$
\end{corollary}
\begin{proof}
	Since $\alpha$ is the center of the block tree of $X$, we have that $X^{\leq \alpha} = X$ and it follows from~\cref{lem:levels_are_preserved} that $\{\alpha\}$ is a orbit of $\Qut(X)$.
	Hence, $\Qut(X) \cong \Qut(X)_\alpha$, and the result follows from~\cref{thm:non-central-cut}.
\end{proof}

\section{Calculations of Quantum Automorphism Groups}

In this section, we apply the machinery built in~\cref{sec:qut_connected}, to calculate quantum automorphism groups of several graph classes. We say that a class of graphs $\calG$ is \emph{hereditary} if it is closed under taking induced subgraphs, i.e., if $X \in \calG$ and $Y$ is an induced subgraph of $X$, then $Y \in \calG$.
If $\calG$ is a hereditary class of graphs such that for any $X,Y \in \calG$, $X \cong_q Y$ if and only if $X \cong Y$, we provide a recipe for computing the quantum automorphism group of any graph $X \in \calG$. 

In summary, the method proceeds by a bottom-up induction on the block tree of a connected graph and repeatedly applies~\cref{thm:block_case},~\cref{thm:non-central-cut}, and~\cref{thm:central-cut}.
In order to apply this method to calculate quantum automorphism groups of some concrete graph class, we need to be able to determine quantum automorphism groups and vertex stabilizers of quantum automorphism groups for biconnected graphs in that class. 

We will give some examples of such graph classes. The first one is forests, for which this recovers results from~\cite{Q_Aut_Trees, meunier2023quantum}. The second one is outerplanar graphs, which is a new family of graphs for which we can describe a procedure computing the quantum automorphism groups. The third one is block graphs, which were treated in~\cite{FMPbis2025} at the same time as the present work was conducted, so that we will simply sketch how our method can recover these results.

%We then give some examples of such graph classes including forests, planar graphs, block graphs, and outerplanar graphs. We give a complete characterization of quantum automorphism groups of biconnected block graphs and outerplanar graphs, which combined with our main result gives a complete characterization of quantum automorphism groups of block graphs and outerplanar graphs. For forests, we are also able to recover the characterization obtained in

We now state and prove our main result. We fix a hereditary graph class $\calG$ such that for $X, Y \in \calG$, $X\cong_q Y$ if an only if $X \cong Y$. Since graphs in $\calG$ are always either isomorphic or quantum isomorphic, the quantum automorphism group of an arbitrary graph in $\calG$ can be expressed in terms of free products and free wreath products of quantum automorphism groups of connected graphs in $\calG$ by an application of~\cref{lem:disjoint-union}. Our main result gives a procedure to compute the quantum automorphism group of a connected graph $X \in \calG$ in terms of vertex stabilizers of quantum automorphism groups of its biconnected components:  

\begin{theorem}
	\label{thm:algorithm}
	The quantum automorphism group of a connected graph $X \in \calG$ can be be determined, from the vertex stabilizers of quantum automorphism groups of its non-central blocks and from the quantum automorphism group of the central block, using the free inhomogeneous wreath product of quantum groups.
\end{theorem}

The following proof is constructive, and can be considered an algorithm that, given a graph $X$, constructs the magic unitary of $\Qut(X)$. One can find a example of how to apply this algorithm in~\cref{fig:outerplanar_example}. 

For explicit constructions of magic unitaries, we refer the reader to~\cref{sec:qut_connected}.
Recall also that by~\cref{prop:free_hom_are_inhom}, free product and free wreath product are special cases of the free inhomogeneous wreath product.
We use this repeatedly in the following proof.

\begin{proof}
	We proceed by induction on the block tree $T(X)$ of $X$.
	In the base step of induction, one has to deal with leaves of $T(X)$, which are biconnected graphs in $\calG$. Let $\alpha \in \calC(X)$. Then, $X^{\leq \alpha}$ is a biconnected subgraph of $\calG$, so that $\Qut(X^{\leq \alpha})_\alpha$ is a vertex stabiliser of the quantum automorphism group of a biconnected graph in $\calG$. 
	This establishes the base of the induction for the leaves of the block tree, and we use this also several times in the induction step.

	For the induction step, we distinguish several cases:
	\begin{itemize}
		\item
		Let $\alpha$ be a non-central cut vertex and let $B_1,\dots,B_n$ be as in~\cref{thm:non-central-cut}.
		By theorem~\cref{thm:non-central-cut}, we have
		\[ \Qut(X^{\leq\alpha})_\alpha \cong \bigast_{i=1}^n\Qut(X^{\leq B_i})_\alpha \fwr \bbS_{k_i}^+. \]
		
		\item
		Let $\alpha$ be the central cut vertex and let $B_1,\dots,B_n$ be as in~\cref{thm:non-central-cut}.
		By theorem~\cref{thm:non-central-cut}, we have
		\[ \Qut(X) \cong \bigast_{i=1}^n\Qut(X^{\leq B_i})_\alpha \fwr \bbS_{k_i}^+. \]
		
		\item
		Let $B$ be a non-central block with a parent cut vertex $\alpha$ and $\alpha_1,\dots, \alpha_n$ as in~\cref{thm:block_case}, part (ii). It follows from theorem~\cref{thm:block_case}, part (ii), that 
		\[ \Qut(X^{\leq B})_\alpha \cong (\Qut(X^{\leq\alpha_1}),\dots,\Qut(X^{\leq\alpha_n}))\fgwr \Qut_c(B)_\alpha. \]
		
		\item
		Let $B$ be the central block and $\alpha_1,\dots, \alpha_n$ as in~\cref{thm:block_case}, part (i). It follows from theorem~\cref{thm:block_case}, part (i), that 
		\[ \Qut(X) \cong (\Qut(X^{\leq\alpha_1}),\dots,\Qut(X^{\leq\alpha_n}))\fgwr \Qut_c(B). \]
	\end{itemize}
	This concludes the induction step and the proof.
	%\Jnote{Shouldn't this proof say something about how we colour the vertices?}
    %\PZnote{If you notice, in each item of the itemize, it is written that ``let ... be as in Theorem XY''. It is already described in the statements of the corresponding theorems how we colour the vertices. I thought we would be repeating a lot of text that is why I did not write it again explicitely.}
    %\Jnote{We refer to Theorem 4.8, which assumes a colouring (in the theorem statement), so we must provide one before invoking the theorem.}
\end{proof}

\subsection{Forests}
A \emph{forest} is a union of trees, i.e., it is a graph with no cycles. In \cite{Q_Aut_Trees, meunier2023quantum}, it was shown that the quantum automorphism group of a forest can be obtained through an iteration of free products and free wreath products by quantum symmetric groups, starting with the trivial group.

If $\calT$ denotes the set of all forests, it is easy to see that $\calT$ is hereditary. It also follows from~\cite[Lemma 2.11]{Q_Aut_Trees} that two forests are isomorphic if and only if they are quantum isomorphic. Also note that biconnected components of forests are edges and that the vertex stabilizer of quantum automorphism group of an edge is the trivial group. 

Hence, it follows from~\cref{thm:algorithm} that the quantum automorphism group of a tree (a connected forest) can be obtained iteratively via free products, free wreath products with the quantum symmetric groups, and free inhomogeneous wreath products. 

One can show that the use of the free inhomogeneous wreath product can be avoided.
Indeed, the only cases where one has the free inhomogeneous wreath product is when the quantum automorphism group of some tree $T$ is expressed as $(\Qut(T^{\leq \alpha})_\alpha, \Qut(T^{\leq \beta})_\beta) \wr_* \bbS_2^+$, where $\alpha\beta \in E(T)$ is the center of $T$.
Evidently, from~\cref{prop:one-prod-to-rule-them-all} if $T^{\leq \alpha}_{\alpha} \cong T^{\leq\beta}_\beta$, we have $(\Qut(T^{\leq \alpha})_\alpha, \Qut(T^{\leq \beta})_\beta) \gwr_* \bbS_2^+ \cong \Qut(X^{\leq \alpha})_{\alpha}\gwr_* \bbS_2^+$.
On the other hand, if $T^{\leq \alpha}_{\alpha} \not\cong T^{\leq\beta}_\beta$, then $(\Qut(T^{\leq \alpha})_\alpha, \Qut(T^{\leq \beta})_\beta) \gwr_* \bbS_2^+ \cong \Qut(X^{\leq \alpha})_{\alpha}\ast \Qut(X^{\leq \beta})_\beta$.

\subsection{Outerplanar graphs.}

Recall that a graph is \emph{planar} if it can be drawn in the plane without intersections.
It is known from~\cite[Corollary A.15]{meunier2023quantum} that two planar graphs are isomorphic if and only if they quantum isomorphic.
Moreover, they constitute a hereditary graph class, so \cref{thm:algorithm} applies to the class of all planar graphs.
However, we do not have a complete understanding of quantum automorphism groups of biconnected planar graphs at the moment. 

In this section, we focus on a subclass of planar graphs, namely the outerplanar graphs.
A graph is \emph{outerplanar} if it can be drawn without intersections in the plane such that all the vertices lie on the outer face.
Outerplanar graphs also form a hereditary class of graphs, so that \cref{thm:algorithm} applies to them.

Every biconnected outerplanar graph $X$ has a unique Hamiltonian cycle \cite[Corollary to Theorem 6]{Syslo-outerplanar}, formed by the boundary of the outer face in any outerplanar drawing of the graph.
%\PNKnote{I think we should mention that $K_2$ is an exception to this. It is probably also the only exception.}
%\Jnote{I am literally citing a theorem from the literature. :P I guess they consider $K_2$ to have a Hamiltonian cycle as well (just not a proper cycle). Do we really need to spell this out?}
If $X$ is a biconnected outerplanar graph, we say that an edge $e \in V(X)$ is \emph{outer} if it belongs to the unique Hamiltonian cycle and \emph{inner} if it does not.
The inner edges are also called \emph{chords}.
We denote the sets of outer and inner edges in $X$ as $E_o(X)$ and $E_i(X)$, respectively.
The following result is folklore, but we include a short proof for completeness.

\begin{proposition}
	\label{prop:WL_outerplanar}
	Let $X$ be a biconnected outerplanar graph.
	Then $2$-WL distinguishes the outer edges of $X$ from the inner edges of $X$.
\end{proposition}
\begin{proof}
	Recall that a vertex set $S \subseteq V(X)$ is called a \emph{$k$-separator} if $|S| = k$ and $\conn(X \setminus S) > \conn(X)$.
	By \cite[Theorem 5.2]{Kiefer-Neuen-decompose}, we have $\wl_X(v,w) \neq \wl_X(x,y)$ if $\{v,w\}$ is a $2$-separator and $\{x,y\}$ is not.
	%\Jnote{TODO: look up the theorem number in the published version of the paper and see if it matches (I don't have access to the published paper right now).}
    %\Jnote{It seems that we have no access to this paper. :-(}
	In a biconnected outerplanar graph $X$, the set $E_i(X)$ of inner edges is exactly the set of edges whose endpoints form a $2$-separator, so the result follows immediately.
\end{proof}

This has a number of direct consequences for quantum automorphism groups of biconnected outerplanar graphs.
First, the quantum automorphism group of $X$ is contained in that of its Hamiltonian cycle.

\begin{corollary}
	\label{cor:outerplanar}
	Let $X$ be a biconnected outerplanar graph, and let $H \subseteq X$ be the unique Hamiltonian cycle.
	Then $\Qut(X) \leq \Qut(H)$.
\end{corollary}
\begin{proof}
	Let $u$ be the fundamental representation of $\Qut(X)$.
	By Proposition \ref{prop:WL_outerplanar} and Lemma \ref{lem:wl-quantum-automorphims}, we have $u_{xy}u_{x'y'} = 0$ if $x$ and $x'$ are connected by an outer edge while $y$ and $y'$ are connected by an inner edge.
	The same equality of course holds if $y$ and $y'$ are not connected at all in $X$ and similarly, $u_{xy}u_{x'y'} = 0$ if $x$ is not connected to $x'$ by an outer edge while $y$ and $y'$ are.
	Therefore, if we denote by $X^{o}$ the graph obtained from $X$ by removing all inner edges, it follows from Lemma \ref{lem:commutation_adjacency} that $u A_{X^{o}} = A_{X^{o}} u$.
	As a consequence, $\Qut(X)\leq \Qut(X^{o}) = \Qut(H)$.
\end{proof}

Second, this yields a characterization of quantum symmetry.

\begin{corollary}
	\label{cor:qut_2con_outerplanar}
	Let $X$ be a biconnected outerplanar graph with $|V(X)| \neq 4$.
	Then $X$ has no quantum symmetry.
\end{corollary}
\begin{proof}
	When $|V(X)| \leqslant 3$, the result follows because $\Qut(X) \leq \bbS_3^+$ and $\bbS_3^+$ is classical.
	Now assume that $|V(X)| \geqslant 5$, and let $H$ be the unique Hamiltonian cycle in $X$.
	It was proved in \cite[Lemma 3.5]{banica2005quantum} that cycles of length $\geqslant 5$ have no quantum symmetry, so it follows from Corollary \ref{cor:outerplanar} that $X$ has no quantum symmetry.
\end{proof}

\begin{remark}
	When $|V(X)| = 4$, there are only two biconnected outerplanar graphs: the square $C_{4}$ and the diamond graph $D$, which is $C_{4}$ plus one diagonal.
	To compute the quantum automorphism groups, we use that $\Qut(X) = \Qut(\overline{X})$.
	We have $\overline{C_4} = K_2 \sqcup K_2$, hence $\Qut(C_{4}) = H_{2}^{2+} = \bbS_{2}^+ \fwr \bbS_{2}^+$.
	Likewise, we have $\overline{D} = K_2 \sqcup K_1 \sqcup K_1$, and therefore $\Qut(D) = \bbS_{2}^+ \ast \bbS_{2}^+$.
\end{remark}

It follows that the vertex stabilizers are classical and subgroups of $S_2 = \bbS_{2}^{+}$.

\begin{corollary}\label{cor:stab_biconnected_outerplanar}
	Let $X$ be a biconnected outerplanar graph.
	Then $\Qut(X)_x \leq \bbS_2^+$ for all $x \in V(X)$.
\end{corollary}
\begin{proof}
	Assume first that $|V(X)| \neq 4$.
	Then $\Qut(X)$ is classical, so by \autoref{cor:outerplanar} it is a subgroup of $\Aut(H) = D_{|H|}$, the dihedral group of order $2|H|$.
	Since every vertex stabilizer of $D_{|H|}$ has order $2$, the result follows.
	
	When $|V(X)| = 4$, we have $X = C_4$ or $X = D$, and a simple case analysis shows that $\Qut(X)_x = \bbS_2^+$ in all cases.
\end{proof}

It follows from the preceding results that the quantum automorphism group of every outerplanar graph can be written as an iterated free inhomogeneous wreath product of quantum symmetric groups and at most one dihedral or cyclic group (for the central block).
However, we do not claim that every quantum group of this form is the quantum automorphism group of an outerplanar graph.

\autoref{fig:outerplanar_example} shows an example of outerplanar graph and its quantum automorphism group.
It shows the quantum automorphism groups of the central block (isomorphic to a $C_4$) and of the subgraphs $X^{\leq \alpha}$ (highlighted in grey circles), where $\alpha$ is one of the four cut vertices in the central block.
The algorithm from \autoref{thm:algorithm} is used to compute the quantum automorphism groups of the four subgraphs and then of the whole graph.

\begin{figure}
	\centering
	\includegraphics{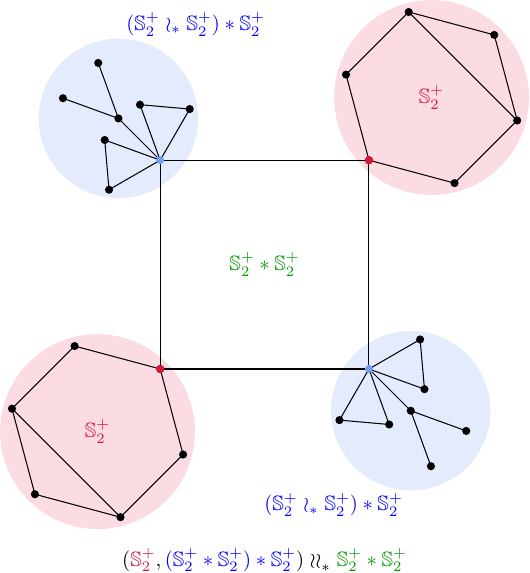}
	\caption{An outerplanar graph and its quantum automorphism group.}
	\label{fig:outerplanar_example}
\end{figure}

\begin{theorem}\label{thm:qut_stab_outerplanar}
    The class $\calS$ of all quantum stabilizers of vertices of outerplanar graphs can be constructed inductively as follows:
    \begin{itemize}
        \itemsep0cm
        \item[(i)] $\one \in \calS$,
        \item[(ii)] if $\mathbb{G}, \mathbb{H} \in \calS$, then $\mathbb{G}\ast\mathbb{H} \in \calS$,
        \item[(iii)] if $\mathbb{G} \in \calS$, then $\mathbb{G}\fwr\mathbb{S}_n^+ \in \calS$, for $n\in\mathbb{N}$.
    \end{itemize}
    Moreover, the class $\calS$ is exactly the class of quantum automorphism groups of (rooted) trees.
\end{theorem}

\begin{proof}
First, we argue that every quantum group in $\calS$ can be realized as a quantum stabilizer of some outerplanar graph. It was proved in~\cite{Q_Aut_Trees} that every quantum group in $\calS$ can be realized as the quantum automorphism group of a rooted tree. A rooted tree is just a tree with one vertex coloured by a unique colour. Since every tree is outerplanar, this gives us what we need.

Let $X$ be an outerplanar graph and $\alpha$ any vertex of $X$.
The vertex $\alpha$ belongs to some block $B$ of $X$. We consider the underlying block-tree to be rooted at $B$.
To calculate $\Qut(X)_\alpha$, we now proceed as in the proof of \cref{thm:algorithm}.
By \cref{cor:stab_biconnected_outerplanar}, we only encouter the free inhomogeneous wreath product with $\mathbb{S}_2^+$. However, in this case this is just the free wreath product.
%\PZnote{Does this proof need more details?}
\end{proof}

\subsection{Block graphs.}

A \emph{block graph} is a graph in which every block is a complete graph. The quantum properties of block graphs were comprehensively investigated in \cite{FMPbis2025}, a study carried out in parallel with the present work. In particular, their quantum automorphism groups were computed in that paper. Nevertheless, the approach developed here can also be used to provide an alternative proof, which we now briefly outline.

First, we note that the class of block graphs is not hereditary. However, the quantum automorphism group of a connected block graph can still be determined using the same inductive procedure described above, provided that we can do the following:
\begin{itemize}
	\item Show that two block graphs are isomorphic if and only if they are quantum isomorphic.
	\item Determine vertex stabilisers of quantum automorphism groups of biconnected block graphs. 
\end{itemize}
The latter is easy to see. Indeed, it is clear that if $B$ is a block of a block graph $X$, then $\Qut(B)$ and $\Qut(B)_\alpha$, for some $\alpha \in V(B)$, are quantum symmetric groups. For the former, we have the following result from \cite[Thm 6.8]{FMPbis2025}

\begin{proposition}
	\label{prop:bolock_graphs_quant_iso}
	Two connected block graphs are quantum isomorphic if and only if they are isomorphic. 
\end{proposition}

\begin{remark}
The result above is a special case of a more general theorem proved in \cite{FMPbis2025}. However, the specific case we require can also be derived more directly using \cite[Theorem 5.2]{freslon2025block} (or a suitable modification of the argument in \cref{lem:gen-res-1}), along with the fact that two vertex-coloured complete graphs are isomorphic if and only if their underlying colourings are preserved under an isomorphism. With these ingredients, the result follows immediately.
\end{remark}

From this it is then easy to derive the following (see also \cite[Thm 6.10]{FMPbis2025}): 

\begin{theorem}
	\label{thm:q_aut_block_graphs}
	The quantum automorphism group of any block graph can be constructed from the trivial group using free products and free wreath products with quantum symmetric groups.
	In particular, the set of quantum automorphism groups of block graphs is the same as the set of quantum automorphism groups of forests.
\end{theorem}

\begin{proof}
	It follows from~\cref{prop:bolock_graphs_quant_iso} that the quantum automorphism group of block graphs can be constructed from quantum automorphism groups of vertex coloured complete graphs via free products, free wreath products with quantum symmetric groups, and free inhomogeneous wreath products. 
	
	However, quantum automorphism groups of vertex coloured complete graphs are free products of quantum symmetric groups. Hence, following along the lines of~\cref{thm:disjoint-union}, we can show that free inhomogeneous wreath products can be replaced by free products and free wreath products with $\bbS_n^+$.  
\end{proof}

\begin{remark}
	Since forests are exactly those graphs where each connected component is $K_2$, the above theorem is a generalisation of~\cite[Theorem 1.1]{Q_Aut_Trees}. %It should be noted that this result was independently obtained by other means in~\cite{FMPbis2025}, where several other quantum properties of block graphs are also worked out.
\end{remark}

\section*{Acknowledgements}

The authors would like to thank Pascal Schweitzer and Ilia Ponomarenko for helpful discussions on $2$-WL for outerplanar graphs.
% \Jnote{Are there other people we would like to thank?}
% \PNKnote{We should probably also mention Simon.}
% \Jnote{Sure, go ahead! (I don't remember what his contribution was, so I don't know what to thank him for.)}

\bibliographystyle{plain}
\bibliography{inhom}

\end{document}